\documentclass[10pt]{amsart}

\usepackage{amsmath,amssymb,amsfonts,amsthm,latexsym,mathrsfs}
\usepackage{graphicx,xcolor,enumerate}
\usepackage{hyperref}
\usepackage{subsupscripts}

\newtheorem{theorem}{Theorem}[section]
\newtheorem{lemma}[theorem]{Lemma}

\newtheorem{corollary}[theorem]{Corollary}

\theoremstyle{definition}

\newtheorem{example}[theorem]{Example}
\theoremstyle{remark}
\newtheorem{remark}[theorem]{Remark}

\numberwithin{equation}{section}

\newcommand{\C}{\mathbb{C}}
\newcommand{\Sp}{\mathbb{S}}

\newcommand{\R}{\mathbb{R}}

\newcommand{\F}{\mathbb{F}}
\newcommand{\Z}{\mathbb{Z}}

\newcommand{\SU}{\operatorname{SU}}

\newcommand{\U}{\operatorname{U}}
\newcommand{\OO}{\operatorname{O}}
\newcommand{\SO}{\operatorname{SO}}
\newcommand{\GL}{\operatorname{GL}}

\newcommand{\Real}{\operatorname{Re}}
\newcommand{\Imaginary}{\operatorname{Im}}
\newcommand{\cl}{\mathrm{cl}}

\newcommand{\G}{\mathbf{G}}
\newcommand{\ftimes}[2]{{\lrsubscripts{\times}{#1}{#2}}}\newcommand{\HOM}{\operatorname{Hom}}
\newcommand{\HOMEO}{\operatorname{Homeo}}

\begin{document}

\title[Euler characteristics of symplectic quotients and $\OO(2)$-spaces]
{Euler characteristics of linear symplectic quotients and $\OO(2)$-spaces}

\author{Carla Farsi}
\address{Department of Mathematics, University of Colorado at Boulder,
    UCB 395, Boulder, CO 80309-0395}
\email{farsi@euclid.colorado.edu}

\author{Hannah Mobley}
\address{Department of Mathematics and Statistics,
    Rhodes College, 2000 N. Parkway, Memphis, TN 38112}
\email{meihe-25@rhodes.edu}

\author{Christopher Seaton}
\address{Department of Mathematics and Statistics,
    Rhodes College, 2000 N. Parkway, Memphis, TN 38112}
\email{seatonc@rhodes.edu}


\thanks{
C.F. was partially supported by the Simons Foundation Collaboration Grant for Mathematicians \#523991.
H.M. was supported by the Rhodes College Summer Fellowship Program and by the National Science Foundation under Grant \#2015553.
C.S. was supported by an AMS-Simons Research Enhancement Grant for PUI Faculty.
}

\begin{abstract}
We give explicit computations of the $\Gamma$-Euler characteristic of several families of orbit space definable translation groupoids. These include the translation groupoids associated to finite-dimensional linear representations of the circle and real and unitary representations of the real $2\times 2$ orthogonal group. In the case of translation groupoids associated to linear symplectic quotients of representations of a arbitrary compact Lie group $G$, we show that unlike the other cases, the $\Gamma$-Euler characteristic depends only on the group and not on the representation.
\end{abstract}


\maketitle

\tableofcontents


\section{Introduction}
\label{sec:intro}

The Euler characteristic $\chi$ is classically defined in the context of triangulable topological spaces as a homotopy invariant.
It can be generalized to definable spaces, a generalization of semialgebraic spaces, sacrificing homotopy invariance but preserving
finite additivity. The Euler characteristic has been generalized to many contexts, including the introduction of many Euler characteristics
for orbifolds, spaces locally modeled on the quotient of a manifold by a finite group. These various orbifold Euler characteristics can be
indexed by a discrete group $\Gamma$, resulting in the $\Gamma$-Euler characteristics $\chi_\Gamma$. Orbifolds can be represented by certain
classes of groupoids; when the groupoid corresponds to a global action of a group $G$ on a topological space $X$, it is called a translation groupoid
and denoted $G\ltimes X$.

In the recent paper \cite{FarsiSeatonEC}, the first and third author gave the definition of the $\Gamma$-Euler characteristics
$\chi_\Gamma$ of sufficiently well-behaved topological groupoids. This extended the work of \cite{GZLMHHigherOrderCompact},
which generalized a large collection of orbifold Euler characteristics to translation groupoids by compact Lie groups that
may have infinite isotropy and hence do not present orbifolds.

For each finitely presented group $\Gamma$, the $\Gamma$-Euler characteristic $\chi_\Gamma$ is an invariant of
\emph{orbit space definable groupoids}, roughly topological groupoids whose orbit spaces and their partitions into (weak) orbit types
can be made definable. The definitions of $\chi_\Gamma$ and orbit space definable groupoids
are recalled in Section~\ref{sec:background}. The $\Gamma$-Euler characteristics $\chi_\Gamma$ are invariant under Morita equivalence of topological groupoids,
multiplicative over Cartesian products of groupoids, and additive over disjoint unions of orbit spaces in the appropriate sense.
With additional mild hypotheses on the orbit space definable groupoids, the $\Gamma$-Euler characteristics can be realized as the usual
(finitely additive) Euler characteristics of definable topological spaces, the \emph{$\Gamma$-inertia spaces}. For an orbit space
definable groupoid $\G$, the $\Gamma$-inertia space is the orbit space of the action groupoid for the conjugation action of $\G$ on
$\operatorname{Hom}(\Gamma,\G)$, which can be understood as the space of homomorphisms from $\Gamma$ into the isotropy groups of $\G$.
When $\Gamma = \Z$, $\operatorname{Hom}(\Z,\G)$ can be identified with the collection of isotropy groups. See \cite{FarsiSeatonEC} for more details.

In this paper, we give explicit computations of $\chi_\Gamma$ for several collections of translation groupoids.
In each case, the translation groupoid is given by a linear action of a compact Lie group $G$ on a finite-dimensional vector space $V$
or a $G$-invariant subset $X$ of $V$. We consider the explicit cases where $G = \Sp^1$ is the circle or $G = \OO(2)$
is the real $2\times 2$ orthogonal group in detail, and also consider the
linear symplectic quotient associated to an arbitrary compact Lie group. A key tool is the use of circle actions on
the orbit space of the groupoid to simplify these computations. To this end, we generalize the localization formula of
\cite{KobayashiFixed} for the Euler characteristic of a compact Riemannian manifold on which a torus acts smoothly to affine definable spaces
with definable torus actions; see Theorem~\ref{thrm:TorusLocalization}.

These computations serve several purposes. First, they give an indication of the degree to which $\chi_\Gamma(G\ltimes X)$
of a translation groupoid $G\ltimes X$ depends on the group $G$ and its action on $X$. In some of the simplest examples one
might initially compute, $\chi_\Gamma(G\ltimes X) = 0$;
however, the cases considered here indicate that $\chi_\Gamma(G\ltimes X)$ generally
depends heavily on both $G$ and its action on $X$; see Remarks~\ref{rem:FreeECS1RepDependence} and \ref{rem:FreeECO2RepDependence}.
In addition, these computations offer counterexamples to properties one might expect from $\chi_\Gamma$; see e.g.
Remark~\ref{rem:FreeECS1RepAdditive}.

Our computations also indicate general properties of the $\Gamma$-Euler characteristics $\chi_\Gamma$. A notable example is the case of linear
symplectic quotients,
the (singular) symplectic quotients associated to unitary representations of compact Lie groups.
As is recalled at the end of Section~\ref{sec:SympQuot}, linear symplectic quotients are local models for symplectic quotients of manifolds,
whose topology is an active area of investigation \cite{DelarueRamacherSchmitt}.
Explicit computations of $\chi_\Gamma$ for translation groupoids
defining linear symplectic quotients corresponding to representations of $\Sp^1$ and $\OO(2)$ suggested that
the $\chi_\Gamma$ of these groupoids have a particularly simple form in general.
This led to the proof of the main result of this paper,
Theorem~\ref{thrm:SympQuot}, which demonstrates that $\chi_\Gamma$ of a linear symplectic quotient by a compact Lie group $G$ depends
only on the group and not on the representation, and is therefore equal to $\chi_\Gamma$ of the group $G$ as a groupoid with one unit.
This differs significantly from many of the other cases considered in this paper, for which $\chi_\Gamma$ depends both on the group
and the action.

The outline of this paper is as follows. In Section~\ref{sec:background}, we recall the background needed to
define $\chi_\Gamma$ focusing on the case of translation groupoids by linear group actions.
Section~\ref{sec:SimpCircleAction} presents the localization formula for the Euler characteristic in the presence of torus
actions as well as specializations that allow us to avoid explicit descriptions of some of the orbit
types in the orbit space $G\backslash X$ in the computation of $\chi_\Gamma(G\ltimes X)$. In Section~\ref{sec:CircleReps}
we compute $\chi_\Gamma(\Sp^1\ltimes V)$ where $V$ is a finite-dimensional representation of the circle as well
as $\chi_\Gamma(\Sp^1\ltimes X)$ for an $\Sp^1$-invariant subspace of such a $V$; Section~\ref{sec:O2Reps} gives similar
computations for a representation of the orthogonal group $\OO(2)$. In Section~\ref{sec:SympQuot}, we
prove Theorem~\ref{thrm:SympQuot}, which determines the $\Gamma$-Euler characteristic
of the linear symplectic quotient associated to a representation of an arbitrary compact Lie group $G$.


\section*{Acknowledgements}

This paper developed from H.M.'s Summer Research Fellowship project and continued research in
the Rhodes College Department of Mathematics and Statistics, and the
authors gratefully acknowledge the support of the department and college for these
activities.
This work was also partially supported by C.F.'s Simons Foundation Collaboration grants \#523991.
H.M. was supported by the Rhodes College Summer Fellowship Program and by the National Science Foundation under Grant \#2015553.
C.S. was supported by an AMS-Simons Research Enhancement Grant for PUI Faculty.
C.S. would like to thank Hans-Christian Herbig, with whom he learned many of the techniques used for
explicit descriptions of symplectic quotients in Section~\ref{sec:SympQuot}.


\section{Background}
\label{sec:background}

In this section, we briefly summarize the background and definitions that we will need. The reader is referred to
\cite[Sec.~2--3]{FarsiSeatonEC} for more details.

Throughout, fix an o-minimal structure \cite{vandenDriesBook} on $\R$ that contains the semialgebraic sets; the o-minimal structure of
semialgebraic sets will be sufficient for much of this paper, but we state some of the results of Section~\ref{sec:SimpCircleAction}
more generally. With the standard topology on $\R^n$, we say that a subspace $A\subseteq\R^n$ that is an element of this structure is
\emph{definable}, and a function $f\colon A\to\R^m$ is a \emph{definable function} if its graph
is a definable subset of $\R^{m+n}$. An \emph{affine definable space} $X$ is a topological space $X$ and a topological embedding
$\iota_X\colon X\to\R^n$ such that $\iota_X(X)$ is a definable set. A subset $A\subset X$ is \emph{definable} if $\iota_X(A)$ is a
definable subset of $\R^n$, and a morphism of affine definable spaces $(X,\iota_X)$ and $(Y,\iota_Y)$ is a continuous function
$f\colon X\to Y$ such that $\iota_Y\circ f\circ \iota_X^{-1}$ is a definable function on $\iota_X(X)$. Every definable subset
$A\subseteq\R^n$ admits a cell decomposition \cite[Ch.~3, (2.11)]{vandenDriesBook}, and the \emph{Euler characteristic} $\chi(A)$
of $A$ is defined to be the sum of the Euler characteristics of its cells \cite[Ch.~4, Sec.~2]{vandenDriesBook}.
If $(X,\iota_X)$ is an affine definable space and $A\subseteq X$ is definable, then the Euler characteristic $\chi(A)$ of $A$ is the Euler
characteristic of $\iota_X(A)$. By \cite[Thrm.~2.2]{Beke}, the Euler characteristic $\chi(A)$ depends only on the underlying topological
space $A$.

For a groupoid $\G_1\rightrightarrows\G_0$, we denote the structure functions as source $s\colon\G_1\to\G_0$,
target $t\colon\G_1\to\G_0$, multiplication $m\colon\G_1\ftimes{s}{t}\G_1\to\G_1$, unit $u\colon\G_0\to\G_1$, and inverse
$i\colon\G_1\to\G_1$. For $x\in\G_0$, we let $\G_x^x$ denote the isotropy group, $\G_x^x := (s,t)^{-1}(x,x)$, and $\G x$
the orbit, $\G x := \{y\in\G_0 : \exists g\in (s,t)^{-1}(x,y)\}$.
A \emph{topological groupoid} is a groupoid such that $\G_0$ and $\G_1$ are topological spaces and
the structure functions are continuous; we will also assume that $\G_0$ and $\G_1$ are Hausdorff and paracompact and
$s$ is an open map. We let $\vert\G\vert$ denote the orbit space and $\pi\colon\G_0\to\vert\G\vert$ the orbit map, which
is a quotient map as $s$ is open. A topological groupoid is \emph{orbit space definable} if it is equipped with a map
$\iota_{\vert\G\vert}\colon\vert\G\vert\to\R^n$ such that $(\vert\G\vert,\iota_{\vert\G\vert})$ is an affine definable space,
$\G_x^x$ is a compact Lie group for each $x\in\G_0$, and the partition of $\vert\G\vert$ into sets of orbits of points
with isomorphic isotropy groups is a finite partition into definable subsets.
Following \cite[Def.~5.6]{PflaumPosthumaTang}, we refer to the elements of this partition of
$\vert\G\vert$ as \emph{weak orbit types}; i.e., the orbits $\G x$ and $\G y$ of respective points $x,y\in\G_0$ are in the
same weak orbit type if $\G_x^x\simeq\G_y^y$ as compact Lie groups.

Let us describe an important class of examples that will be the focus of much of this paper; we summarize the relevant definitions
and results here and refer the reader to \cite[Sec.~2.2]{FarsiSeatonEC} and \cite{ChoiParkSuhSemialgSlices,ParkSuhLinEmbed} for more details and references.
Let $G$ be a compact Lie group. Then $G$ has a faithful representation and hence is isomorphic to a compact linear algebraic group.
This identifies $G$ with a \emph{semialgebraic group}, a semialgebraic subset of some Euclidean space that is a topological group
whose multiplication and inverse maps are definable functions in the o-minimal structure of semialgebraic sets. It is moreover
a \emph{semialgebraic linear group}, i.e., is semialgebraically isomorphic to a semialgebraic subgroup of $\GL_m(\R)$. The corresponding
semialgebraic linear group is unique and hence gives $G$ the unique structure of an \emph{affine definable topological group} with
respect to any o-minimal structure that contains the semialgebraic sets. That is, with respect to the above embedding of $G$
into $\GL_m(\R)\subset\R^{m^2}$, the multiplication and inverse maps are morphisms of the affine definable spaces $G\times G\to G$ and $G\to G$,
respectively. An \emph{affine definable $G$-space $X$} is then an affine definable space $(X,\iota_X)$ where $X$ is a topological
$G$-space and the action map $G\times X\to X$ is a morphism of affine definable spaces; we then say that the $G$-action is a
\emph{definable $G$-action}. It follows that $X$ admits a \emph{definable proper quotient}, i.e., $G\backslash X$ admits a unique
structure, up to definable homeomorphism, as an affine definable space with respect to which the orbit map $X\to G\backslash X$
is a morphism of affine definable spaces. If the o-minimal structure is that of semialgebraic sets and $X\subseteq\R^n$ is semialgebraic,
i.e., $\iota_X$ is the identity, $X$ is called a \emph{semialgebraic $G$-set}. In this case, the translation groupoid $G\ltimes X$
is an orbit space definable groupoid by \cite[Cor.~3.6]{FarsiSeatonEC}. In particular, if $V$ is a finite-dimensional linear representation
of $G$ and $X$ is a semialgebraic subset of $V$, then $G\ltimes X$ is an orbit space definable groupoid.

If $\G$ is an orbit space definable groupoid and $\Gamma$ is a finitely presented discrete group, the \emph{$\Gamma$-Euler characteristic of $\G$}
\cite[Def.~3.8]{FarsiSeatonEC} is
\begin{equation}
\label{eq:GamEC}
    \chi_\Gamma(\G) = \int_{\vert\G\vert} \chi\big(\G_x^x\backslash\HOM(\Gamma,\G_x^x)\big)\, d\chi(\G x),
\end{equation}
where $\HOM(\Gamma,\G_x^x)$ is the space of group homomorphisms from $\Gamma$ to the isotropy group $\G_x^x$ with the compact-open topology,
the integral is that with respect to the Euler characteristic \cite{CurryGhristEC,ViroEC}, and the action of
$\G_x^x$ on $\HOM(\Gamma,\G_x^x)$ is by pointwise conjugation. Because the integrand
$\chi\big(\G_x^x\backslash\HOM(\Gamma,\G_x^x)\big)$ depends only on the isomorphism class of $\G_x^x$, this can be reformulated
as follows. Let
\[
    \vert\G\vert = \vert\G\vert_1\sqcup\vert\G\vert_2\sqcup\cdots\sqcup\vert\G\vert_r
\]
denote the partition of the orbit space into weak orbit types, i.e., orbits of points with isomorphic isotropy groups.
For each $i$, let $G_i$ denote the isotropy group of a choice of point whose orbit is in $\vert\G\vert_i$. Then
\begin{equation}
\label{eq:GamECAsSumIsotropy}
    \chi_\Gamma(\G) = \sum\limits_{i=1}^r \chi\big(\vert\G\vert_i\big) \chi\big(G_i\backslash\HOM(\Gamma,G_i)\big).
\end{equation}
In the same way, if
\[
    \vert\G\vert = X_1\sqcup X_2\sqcup\cdots\sqcup X_s
\]
is a partition of $\vert\G\vert$ into definable subsets that are each contained in a weak orbit type and $G_i$ denotes a choice of
isotropy group of a point whose orbit is in $X_i$, then
\begin{equation}
\label{eq:GamECAsSumFinerPartition}
    \chi_\Gamma(\G) = \sum\limits_{i=1}^s \chi(X_i) \chi\big(G_i\backslash\HOM(\Gamma,G_i)\big).
\end{equation}

Finally, let us fix some notation that is used throughout this paper. If $G$ is a compact Lie group, $X$ is a $G$-space,
and $H\leq G$, we let $(H)$ denote the $G$-conjugacy class of $H$ in $G$ and define
$X^H := \{ x\in X : hx = x \; \forall h \in H\}$, the set of points fixed by $H$;
$X_H := \{ x\in X : G_x = H \}$, points with isotropy group $H$; and
$X_{(H)} := \{ x\in X : G_x \in (H) \}$, points with isotropy group conjugate to $H$.
The set $X_{(H)}$ is called the \emph{orbit type} associated to the conjugacy class $(H)$. Because $X_{(H)}$ is $G$-invariant,
we will use the term ``partition into orbit types" to describe the partition of $X$ as well as the induced partition of the orbit space $G\backslash X$.
Note that if $G\ltimes X$ is an orbit space definable groupoid, then an orbit type of $\vert G\ltimes X\vert = G\backslash X$ is
obviously contained in a weak orbit type.

For a topological space $X$, we let $\HOMEO(X)$ denote the group of homeomorphisms $X\to X$.
If $\rho\colon G\to\GL(V)$ is a linear representation of $G$, $g\in G$, and $W\subseteq V$ is a $\rho(g)$-invariant subspace, we use the
notation $\rho(g)_{|W}$ to denote the restriction of $\rho(g)$ to $W$. If $W$ is $\rho(g)$-invariant for all $g\in G$, we use
$\rho_{|W}$ to denote the resulting restricted representation $G\to\GL(W)$ defined by $g\mapsto\rho(g)_{|W}$.
If $X$ is a $\rho(g)$-invariant subset of $V$, we also
use $\rho(g)_{|X}$ to denote the restriction of $\rho(g)$ to $X$ and,
if $X$ is $\rho(g)$-invariant for all $G$, $\rho_{|X}$
to denote the map $G\to\HOMEO(X)$ given by $g\mapsto\rho(g)_{|X}$.
For a positive integer $n$, we let $\underline{n} = \{1, 2, \ldots, n\}$,
$R(n)$ denote the group of $n$th roots of unity, $D_{2n}$ denote the dihedral group with $2n$ elements,
and $\F_n$ denote the free group with $n$ generators.


\section{Localization to the fixed points of a torus action}
\label{sec:SimpCircleAction}

An important tool we will use throughout this paper is the simplification of Euler characteristic computations in the presence of circle actions.
This is a consequence of the following localization theorem for the Euler characteristic, which is well-known in the context of torus actions on
manifolds; see \cite[Remark on p.~64]{KobayashiFixed} and \cite[Thrm.~9.3]{GoertschesZoller}. In this section, we will prove this result and
then elaborate consequences for the computation of $\chi_\Gamma$ in specific circumstances.

\begin{theorem}
\label{thrm:TorusLocalization}
Let $T = (\Sp^1)^\ell$ be a torus and let $X$ be an affine definable space that admits a definable $T$-action.
Then $\chi(X) = \chi(X^T)$. In particular, $\chi(X) = 0 $ if $X^T = \emptyset$.
\end{theorem}
\begin{proof}
As $T$ is compact, $X$ admits a definably proper quotient $\pi\colon X\to T\backslash X$ so that $T\backslash X$ is an affine definable
space and $\pi$ is a morphism of affine definable spaces; see \cite[pp.2348--9]{FarsiSeatonEC}.
By the trivialization theorem \cite[Ch.~9 (1.7)]{vandenDriesBook}, there is a partition $T\backslash X = Y_1\sqcup\cdots\sqcup Y_r$ of
$T\backslash X$ into definable subsets such that $\pi^{-1}(Y_i)$ is definably homeomorphic to $F_i\times Y_i$ for a definable set $F_i$
that is itself definably homeomorphic to $\pi^{-1}(x)$ for each $x\in Y_i$. Reindex the $Y_i$ if necessary so that $F_i$ is a single point
for $i\leq k$ and not a singleton for $i > k$, where we may have $k = 0$ or $k = r$. As the action map $T\times X\to X$ induces for each $x\in X$
a definable homeomorphism of $T_x\backslash T$ onto the orbit $\pi^{-1}(x)$ of $x$ in $X$, it follows that each $F_i$ is definably homeomorphic to
$T_i\backslash T$ where $T_i$ is the isotropy group of a point $x\in \pi^{-1}(Y_i)$. Therefore, $T_i = T$ for $i\leq k$ and $T_i \lneq T$ for
$i > k$; that is, $X^T = \bigsqcup_{i=1}^k \pi^{-1}(Y_i)$. If $i > k$, then as $F_i\simeq T_i\backslash T$ is a positive-dimensional compact connected
abelian Lie group and hence a nontrivial torus, we have by the multiplicativity of $\chi$ \cite[Ch.~4, Cor.~(2.11)]{vandenDriesBook} that
$\chi\big(\pi^{-1}(Y_i)\big) = \chi(F_i\times Y_i) = 0$. It follows that
\begin{align*}
    \chi(X)     &=      \sum\limits_{i=1}^r \chi\big(\pi^{-1}(Y_i)\big)
                \\&=    \sum\limits_{i=1}^k \chi\big(\pi^{-1}(Y_i)\big)
                \\&=    \chi(X^T).
    \qedhere
\end{align*}
\end{proof}

Combining Theorem~\ref{thrm:TorusLocalization} with Equation~\eqref{eq:GamECAsSumFinerPartition} yields the following.

\begin{corollary}
\label{cor:SimpTorusActionOSDGroupoid}
Let $\G$ be an orbit space definable groupoid, let
\[
    \vert\G\vert = X_1\sqcup X_2\sqcup\cdots\sqcup X_s
\]
be a partition of $\vert\G\vert$ into definable sets such that each $X_i$ is contained in a weak orbit type, and let $G_i$ be the isotropy
group of a point whose orbit is in $X_i$. Let $T_i$ be a (possibly trivial) torus for each $i$, suppose that each $X_i$ admits a definable
$T_i$-action, and let $I = \{ i : X_i^{T_i} \neq \emptyset\}$. Then
\begin{equation}
\label{eq:SimpCircleActionOSDGroupoid}
    \chi_\Gamma(\G) = \sum\limits_{i\in I} \chi(X_i^{T_i}) \chi\big(G_i\backslash\HOM(\Gamma,G_i)\big).
\end{equation}
\end{corollary}

In the sequel, we will primarily apply Theorem~\ref{thrm:TorusLocalization} in situations where $T = \Sp^1$, and it will be convenient to state a few
specializations to this case. As the only proper closed subgroups of $\Sp^1$ are finite, we have the following.

\begin{corollary}
\label{cor:S1ActionMain}
Let $X$ be an affine definable space that admits a definable $\Sp^1$-action such that the isotropy group of each $x\in X$ is finite.
Then $\chi(X) = 0$.
\end{corollary}

We will need the following specific application of Corollary~\ref{cor:S1ActionMain}.

\begin{corollary}
\label{cor:SimpCircleActionRepMulti}
Let $G$ be a compact Lie group and $\rho\colon G\to\GL(V)$ a finite-dimensional linear $G$-representation.
Suppose $X \subseteq V$ is a $G$-invariant definable subset, $H$ is the isotropy group of a point in $X$, and $Z\subset X_{(H)}$ is
$G$-invariant and definable. Let $Y = Z\cap X_H$.
Suppose further that $T \leq C_{\HOMEO(Y)}\Big(\rho\big(N_G(H)\big)_{|Y}\Big)$ is a subgroup of the centralizer of $\rho\big(N_G(H)\big)_{|Y}$
in $\HOMEO(Y)$ such that $T\simeq\Sp^1$, $T$ acts on $Y$ with finite isotropy, and the intersection of $T$ and $\rho\big(N_G(H)\big)_{|Y}$ in
$\HOMEO(Y)$ is finite. Then
\[
    \chi_\Gamma(G\ltimes X) =   \chi_\Gamma\Big( G\ltimes \big(X\smallsetminus Z\big)\Big).
\]
\end{corollary}
\begin{proof}
The orbit type in $\vert G\ltimes X\vert$ associated to $H$ is given by $G\backslash X_{(H)}$, which is homeomorphic to
$N_G(H)\backslash X_H$ by \cite[Thrm.~4.3.10]{PflaumBook}. Restricting this homeomorphism to $G\backslash Z$ yields a homeomorphism to $N_G(H)\backslash Y$.
As $T$ commutes with $\rho\big(N_G(H)\big)_{|Y}$, the $T$-action descends to a definable action on $N_G(H)\backslash Y$, and the fact
that $T$ has finite intersection with $\rho\big(N_G(H)\big)_{|Y}$ in $\HOMEO(Y)$ implies that the $T$-action on $N_G(H)\backslash Y$
has finite isotropy.
\end{proof}

Continuing to let $V$ denote a linear $G$-representation,
let $D = \{ e^{\sqrt{-1}\,\theta}I\in\GL(V):\theta\in\R\}\simeq\Sp^1$
denote the subgroup of $\GL(V)$ that acts on $V$ as scalar multiplication by $e^{\sqrt{-1}\,\theta}$. Because $D$ is in the center of
$\GL(V)$, for each isotropy group $H$ of the $G$-action on $V$, we have that $V_H$, $V_{(H)}$, and $V^H$ are $D$-invariant.
Then as $D$ fixes no points in $V$ except the origin, we have the following.

\begin{corollary}
\label{cor:SimpCircleActionDiag}
Let $G$ be a compact Lie group, $\rho\colon G\to\GL(V)$ a finite-dimensional linear $G$-representation, and let
$D = \{ e^{\sqrt{-1}\,\theta}I\in\GL(V):\theta\in\R\}\simeq\Sp^1$
denote the subgroup of $\GL(V)$ that acts on $V$ as scalar multiplication by $e^{\sqrt{-1}\,\theta}$. Let
$X \subseteq V$ be $G$-invariant, $D$-invariant, and definable; let
\[
    G\backslash X = (G\backslash X)_1 \sqcup (G\backslash X)_2 \sqcup\cdots\sqcup (G\backslash X)_s
\]
denote the partition of $G\backslash X$ into orbit types; and for each $i$, let $H_i$ denote the isotropy group of a choice of point
with orbit in $(G\backslash X)_i$. For a fixed $j$, suppose $0\notin X_{H_j}$ and the intersection of $D_{|X_{H_j}}$ with
$\rho\big(N_G(H_j)\big)_{|X_{H_j}}$ in $\HOMEO(X_{H_j})$ is finite.
Then for any finitely presented $\Gamma$, $(G\backslash X)_j$ does not contribute to the $\Gamma$-Euler characteristic of $G\ltimes X$, i.e.,
\[
    \chi_\Gamma(G\ltimes X) =   \sum\limits_{\substack{i=1\\ i\neq j}}^s
                                \chi\big( (G\backslash X)_i\big) \chi\big(G_i\backslash\HOM(\Gamma,G_i)\big).
\]
\end{corollary}


\section{Computations for $\Sp^1$-representations}
\label{sec:CircleReps}

In this section, we compute the $\Gamma$-Euler characteristics of translation groupoids associated to linear representations of the
circle. We additionally consider some cases of invariant subsets of representations.

Let $V$ be a finite-dimensional hermitian vector space, let $\U(V)$ denote the group of unitary transformations $V\to V$, and let
$\rho\colon\Sp^1\to\U(V)$ be a unitary $\Sp^1$-representation.
Choosing a basis with respect to which the action is diagonal,
let $(x_1,\ldots,x_n)$ be coordinates for $V$ with respect to this basis, and then the action of $\Sp^1$ is described by a weight vector
$(a_1,\ldots,a_n)\in\Z^n$. Specifically, the action is given by
\[
    z(x_1,\ldots,x_n)   =   (z^{a_1}x_1,\ldots,z^{a_n}x_n),
    \qquad z\in\Sp^1,
    \quad (x_1,\ldots,x_n)\in V.
\]
The isotropy group of a point $(x_1,\ldots,x_n)\in V$ is determined by the coordinates $x_i$ that are nonzero, and the possible
isotropy groups are either $\Sp^1$ or the finite group $R(m)$ of $m$th roots of unity isomorphic to $\Z/m\Z$, $m \geq 1$.

Let $\underline{n} = \{1,2,\ldots,n\}$. We partition $V$ by defining for $I\subseteq\underline{n}$ the set
\[
    V_I = \{ (x_1,\ldots,x_n)\in V : x_i \neq 0 \:\forall i \in I \text{ and } x_i = 0:\forall i \notin I \}.
\]
We also define $\mathbf{a}_I = \{ a_i : i \in I\}$ to be the set of weights corresponding to elements of $I$ and set $\mathbf{a}_{\emptyset} = \{0\}$.
Then each $V_I$ is a subset of an orbit type in $V$ so that $\{ V_I : I\subseteq\underline{n}\}$ is a partition
subordinate to the partition of $V$ into orbit types. Specifically, the isotropy group of $V_I$ is given by
the group $R\big(\gcd(\mathbf{a}_I)\big)$ of $\gcd(\mathbf{a}_I)$th roots of unity, where in the case $\mathbf{a}_I = \{0\}$, we use the common
convention that $\gcd(\{0\}) = 0$ and adopt the convention that $R(0) = \Sp^1$. Note that the closure $V_I^\cl$ of $V_I$ is the coordinate
subspace of $V$ given by
\[
    V_I^\cl = \{ (x_1,\ldots,x_n)\in V : x_i = 0:\forall i \notin I \}.
\]
Note further that $V^{\Sp^1} = \bigsqcup_{I\subseteq\underline{n} \colon \mathbf{a}_I =\{0\}} V_I$.

\begin{theorem}
\label{thrm:GamECS1Rep}
Let $\rho\colon\Sp^1\to\U(V)$ be a unitary $\Sp^1$-representation with weight vector $\mathbf{a} = (a_1,\ldots,a_n)\in\Z^n$
and let $\Gamma$ be a finitely presented discrete group. Then
\[
    \chi_\Gamma(\Sp^1\ltimes V)
    =   \chi\big(\HOM(\Gamma,\Sp^1)\big)
        -\sum\limits_{\substack{i=1 \\ a_i\neq 0}}^n \chi\big(\HOM(\Gamma,\Z/a_i\Z)\big).
\]
\end{theorem}
\begin{proof}
Consider the partition of $\Sp^1\backslash V$ into the sets $\{ \Sp^1\backslash V_I : I\subseteq\underline{n} \}$.
Each element $\Sp^1\backslash V_I$ is contained in the weak orbit type of the group $R\big(\gcd(\mathbf{a}_I)\big)$.
In particular, $\Sp^1\backslash V^{\Sp^1} = V^{\Sp^1} = \bigsqcup_{I\subseteq\underline{n} \colon \mathbf{a}_I =\{0\}} V_I$,
and each $\Sp^1\backslash V_I$ with $\mathbf{a}_I\neq\{0\}$ is contained in the weak orbit type of the finite group $R\big(\gcd(\mathbf{a}_I)\big)$.
As $\Sp^1$ is abelian so that conjugation is trivial, we can rewrite
Equation~\eqref{eq:GamECAsSumFinerPartition} in this context as
\begin{equation}
\begin{split}
\label{eq:GamECS1Rep}
    \chi_\Gamma(\Sp^1\ltimes V)
    &=          \chi\big(\HOM(\Gamma,\Sp^1)\big) \chi\big( \Sp^1\backslash V^{\Sp^1}\big)
    \\&\quad    + \sum\limits_{I\subseteq\underline{n} \colon \mathbf{a}_I \neq\{0\}}
                \chi\big(\HOM(\Gamma,\Z/\gcd(\mathbf{a}_I)\Z)\big) \chi\big(\Sp^1\backslash V_I\big).
\end{split}
\end{equation}
The set $\Sp^1\backslash V^{\Sp^1} = V^{\Sp^1}$ is a complex subspace of $V$, hence a single even-dimensional cell with Euler characteristic $1$,
so that the first term in Equation~\eqref{eq:GamECS1Rep} reduces to $\chi\big(\HOM(\Gamma,\Sp^1)\big)$.

For each $i$ such that $a_i\neq 0$, $V_{\{i\}} \simeq\C\smallsetminus\{0\}$, and $z\in\Sp^1$ acts on $V_{\{i\}}$ as multiplication by $z^{a_i}$.
It follows that the orbit space $\Sp^1\backslash V_{\{i\}}$ is an open interval with isotropy group $R(|a_i|)\simeq\Z/a_i\Z$, and the corresponding
term in \eqref{eq:GamECS1Rep} is $-\chi\big(\HOM(\Gamma,\Z/a_i\Z)\big)$.

To complete the proof, we claim that for each $I$ that contains at least two elements and such that $\mathbf{a}_I\neq\{0\}$,
we have $\chi\big(\Sp^1\backslash V_I\big) = 0$.
Let $I = \{i_1,\ldots,i_k\}$ with $k\geq 2$ and at least one $a_{i_j}\neq 0$ so that $\mathbf{a}_I\neq\{0\}$.
If the $a_{i_1},\ldots,a_{i_k}$ do not all coincide, then the circle action on $V_I$ by scalar multiplication, i.e., the restriction of
the action with weight vector $(1,\ldots,1)$ on the vector space $V_I^\cl$, commutes with the $\Sp^1$-action on $V_I$ and has finite intersection
with $\rho(\Sp^1)_{|V_I}$. If $a_{i_1} = a_{i_2} = \cdots = a_{i_k}$, then the same holds for the circle action with weight
vector $(1,\ldots,1,-1)$ on $V_I^\cl$. In either case, $\chi\big(\Sp^1\backslash V_I\big) = 0$ by Corollary~\ref{cor:SimpCircleActionRepMulti}.
\end{proof}

\begin{remark}
\label{rem:GamECS1RepReal}
The only real irreducible representation of $\Sp^1$ that does not admit a unitary structure is the trivial $1$-dimensional representation.
Hence, if $V$ is a finite-dimensional real linear representation of $\Sp^1$ then $V = W\oplus V^{\Sp^1}$ where $W$ admits a unitary structure and
$V^{\Sp^1}$ is a real vector space. Say $W$ has weight vector $(a_1,\ldots,a_n)\in\Z^n$ and note that each $a_i\neq 0$ by construction. As groupoids,
$\Sp^1\ltimes V \simeq (\Sp^1\ltimes W)\times V^{\Sp^1}$ where $V^{\Sp^1}$ is the base groupoid consisting only of units.
Then by the multiplicativity of $\chi_\Gamma$ \cite[Lem.~4.17]{FarsiSeatonEC}, we have that the $\Gamma$-Euler characteristic
$\chi_\Gamma(\Sp^1\ltimes V)$ is given by
\begin{align*}
    (-1)^d \chi_\Gamma(\Sp^1\ltimes W)
        + (-1)^d\chi\big(\HOM(\Gamma,\Sp^1)\big)
        + (-1)^{d+1}\sum\limits_{i=1}^n \chi\big(\HOM(\Gamma,\Z/a_i\Z)\big),
\end{align*}
where $d$ is the real dimension of $V^{\Sp^1}$.
\end{remark}

Recall that $\F_\ell$ denotes the free group with $\ell$ generators.
When $\Gamma = \Z^\ell$ or $\F_\ell$, we have $\HOM(\Gamma,\Sp^1)\simeq(\Sp^1)^\ell$ and
$\HOM(\Gamma,\Z/a_i\Z) \simeq (\Z/a_i\Z)^\ell$, implying the following.

\begin{corollary}
\label{cor:FreeECS1Rep}
Let $\rho\colon\Sp^1\to\U(V)$ be a unitary $\Sp^1$-representation with weight vector $\mathbf{a} = (a_1,\ldots,a_n)\in\Z^n$. Then
\[
    \chi_{\Z^\ell}(\Sp^1\ltimes V) = \chi_{\F_\ell}(\Sp^1\ltimes V)
    =   -\sum\limits_{i=1}^n \vert a_i \vert^\ell.
\]
If $V$ is a real representation so that $V = W\oplus V^{\Sp^1}$ where $W$ is unitary with weight vector
$\mathbf{a} = (a_1,\ldots,a_n)\in\Z^n$ and each $a_i\neq 0$, then
\[
    \chi_{\Z^\ell}(\Sp^1\ltimes V) = \chi_{\F_\ell}(\Sp^1\ltimes V)
    =   (-1)^{d+1}\sum\limits_{i=1}^n \vert a_i \vert^\ell,
\]
where $d$ is the real dimension of $V^{\Sp^1}$.
\end{corollary}

\begin{remark}
\label{rem:FreeECS1RepAdditive}
From Corollary~\ref{cor:FreeECS1Rep}, we see that when $\Gamma = \Z^\ell$ or $\F_\ell$, if $V_1$ and $V_2$ are two unitary
representations of $\Sp^1$, then
\[
    \chi_\Gamma\big(\Sp^1\ltimes (V_1\oplus V_2)\big)
        =   \chi_\Gamma(\Sp^1\ltimes V_1) + \chi_\Gamma(\Sp^1\ltimes V_2).
\]
By Theorem~\ref{thrm:GamECS1Rep} and Remark~\ref{rem:GamECS1RepReal},
this identity does not hold for $\Gamma$ such that $\chi\big(\HOM(\Gamma,\Sp^1)\big) \neq 0$, nor for
$\Gamma = \Z^\ell$ or $\F_\ell$ when the $V_i$ are real representations. We will see in Section~\ref{sec:O2Reps}
that this also does not hold for representations of $\OO(2)$, even when restricting to unitary representations;
see Theorem~\ref{thrm:GamECO2Rep}.
\end{remark}

\begin{remark}
\label{rem:FreeECS1RepDependence}
Regarding the question of how much the $\Gamma$-Euler characteristics $\chi_\Gamma$ depend on the representation, let us observe that if
$V_1$ and $V_2$ are unitary representations of the circle such that
$\chi_\Gamma(\Sp^1\ltimes V_1) = \chi_\Gamma(\Sp^1\ltimes V_2)$ for $\Gamma = \Z^\ell$ with $\ell = 1,\ldots, n$,
then the absolute values of the weights of $V_1$ and $V_2$ coincide up to permuting weights. This follows from
Corollary~\ref{cor:FreeECS1Rep} and the facts that the power sums generate the elementary symmetric polynomials
and the elementary symmetric polynomials separate orbits of points in $\C^n$.
If the $V_i$ are real representations, then the same holds, and the sign of any $\chi_{\Z^\ell}(\Sp^1\ltimes V_i)$
determines the parity of the real dimensions of the $V_i^{\Sp^1}$. Of course, if the real dimensions of the $V_i^{\Sp^1}$ have the same parity
and the absolute values of the weights are the same for both, then $\chi_\Gamma(\Sp^1\ltimes V_1) = \chi_\Gamma(\Sp^1\ltimes V_2)$
for any $\Gamma$ by Theorem~\ref{thrm:GamECS1Rep} and Remark~\ref{rem:GamECS1RepReal}.
\end{remark}

A slight modification of the proof of Theorem~\ref{thrm:GamECS1Rep} also yields the following.

\begin{corollary}
\label{cor:GamECS1Set}
Let $\rho\colon\Sp^1\to\U(V)$ be a unitary $\Sp^1$-representation with weight vector $\mathbf{a} = (a_1,\ldots,a_n)\in\Z^n$.
Suppose $X\subset V$ is an $\Sp^1$-invariant definable subset with the following property:
For each $I = \{i_1,\ldots,i_k\}$ with $k\geq 2$ and $\mathbf{a}_I \neq \{0\}$ such that $V_I\cap X\neq\emptyset$,
if $\{ a_i : i\in I\}$ has more than one element, then $X\cap V_I$ is invariant under the circle action by scalar multiplication,
and if $\{ a_i : i\in I\}$ is a singleton, then $X\cap V_I$ is invariant under a circle action with weight $(\pm 1, \ldots, \pm 1)$
where each sign occurs at least once. Then
\begin{align*}
    \chi_\Gamma(\Sp^1\ltimes X)
    &=          \chi\big(\HOM(\Gamma,\Sp^1)\big) \chi\big(X\cap V^{\Sp^1}\big)
    \\&\quad    +\sum\limits_{\substack{i=1 \\ a_i\neq 0}}^n
                \chi\big(G\backslash(X\cap V_{\{i\}})\big) \chi\big(\HOM(\Gamma,\Z/a_i\Z)\big).
\end{align*}
\end{corollary}

Note that the hypotheses of Corollary~\ref{cor:GamECS1Set} are satisfied for any $X$ defined by restrictions on the moduli of the coordinates
of elements of $V$. As an example, let $S = \{(x_1,\ldots,x_n)\in V: |x_1|^2 + \cdots + |x_n|^2 = 1 \}\simeq\Sp^{2n-1}$ denote the unit sphere in $V$.
Then $S\cap V^{\Sp^1}$ is empty or an odd-dimensional sphere so that $\chi\big(S\cap V^{\Sp^1}\big) = 0$,
and if $a_i\neq 0$, then $G\backslash(S\cap V_{\{i\}})$ is a point so that $\chi\big(G\backslash(S\cap V_{\{i\}})\big) = 1$, yielding
\[
    \chi_\Gamma(\Sp^1\ltimes S)
    =   \sum\limits_{\substack{i=1 \\ a_i\neq 0}}^n \chi\big(\HOM(\Gamma,\Z/a_i\Z)\big).
\]
Similarly, if $B = \{(x_1,\ldots,x_n)\in V: |x_1|^2 + \cdots + |x_n|^2 \leq 1 \}$ is the unit ball in $V$, then
$B\cap V^{\Sp^1}$ is a point or closed ball so that $\chi\big(B\cap V^{\Sp^1}\big) = 1$, and
$G\backslash(B\cap V_{\{i\}})$ is a half-open interval with $\chi\big(G\backslash(B\cap V_{\{i\}})\big) = 0$,
yielding
\[
    \chi_\Gamma(\Sp^1\ltimes B)
    =   \chi\big(\HOM(\Gamma,\Sp^1)\big).
\]
Another important case of an $\Sp^1$-invariant subset of a unitary $\Sp^1$-representation $V$ is the zero fiber of the moment map, which
will be considered in Section~\ref{sec:SympQuot} in the general setting of a unitary representation of an arbitrary compact Lie group.


\section{Computations for $\OO(2)$-representations}
\label{sec:O2Reps}

Throughout this section, $G = \OO(2) = \OO(2,\R)$, which consists of $\SO(2)$, the rotations
\[
    r_\theta = \begin{pmatrix}
        \cos{\theta}   &   -\sin{\theta}   \\
        \sin{\theta}   &   \cos{\theta}
    \end{pmatrix}
\]
and $\OO(2)\smallsetminus\SO(2)$, the reflections
\[
    s_\theta = \begin{pmatrix}
        \cos{\theta}   &   \sin{\theta}   \\
        \sin{\theta}   &   -\cos{\theta}
    \end{pmatrix},
\]
for $\theta\in\R$. We briefly recall the representation theory of $G$ and refer the reader to
\cite[Theorem~7.2.1]{PalmThesis} and \cite[Section~11.2]{Knightly} for more information.

The irreducible unitary representations of $\OO(2)$ are those induced from the $1$-dimensional irreducible representations
of $\SO(2)\simeq\Sp^1$. For $a\in\Z$, let $\epsilon_a$ denote the $1$-dimensional circle representation given by multiplication by $z^a$
for $z \in\Sp^1$ (i.e., the representation with weight vector $(a)$ in the language of Section~\ref{sec:CircleReps}).
Let $\tau_a$ denote the induced representation of $\OO(2)$. Then $\tau_0$ splits into the trivial $1$-dimensional representation and the
$1$-dimensional representation $\det$ with kernel $\SO(2)$ whose action is multiplication by the determinant $\pm 1$.
For $a\neq 0$, $\tau_a$ is irreducible and isomorphic to $\tau_{-a}$. We will therefore only consider $\tau_\alpha$ for $\alpha > 0$,
using the notation $\alpha$  to emphasize that the integer must be positive. Then $\{ \det, \tau_\alpha : \alpha > 0 \}$ is a complete
list of unitary irreducible representations of $\OO(2)$. Expressed as matrices, $\det\colon r_\theta\mapsto (1)$ and $\det\colon s_\theta\mapsto (-1)$,
while for $\alpha > 0$,
\[
    \tau_\alpha\colon r_\theta\mapsto \begin{pmatrix}
        e^{\sqrt{-1}\,\alpha\theta}   &   0   \\
        0   &   e^{-\sqrt{-1}\,\alpha\theta}
    \end{pmatrix},
    \qquad
    \tau_\alpha\colon s_\theta\mapsto \begin{pmatrix}
        0   &   e^{\sqrt{-1}\,\alpha\theta}   \\
        e^{-\sqrt{-1}\,\alpha\theta}   &  0
    \end{pmatrix}.
\]
In either $\det$ or $\tau_\alpha$, the isotropy group of the origin is of course $\OO(2)$. The isotropy group of any nonzero point of $\det$
is $\SO(2)$. In $\tau_\alpha$ with coordinates $(x_1,x_2)$, the kernel is the group $R(\alpha)\simeq\Z/\alpha\Z$ of $\alpha$th roots of unity in $\SO(2)$.
If $|x_1| = |x_2|$, then $x_1 = e^{\sqrt{-1}\,\alpha\theta}x_2$ has solutions corresponding to $\alpha$ distinct reflections so that $(x_1, x_2)$
has isotropy isomorphic to $D_{2\alpha}$, the dihedral group with $2\alpha$ elements.


\subsection{Computations for $\OO(2)$-representations with general $\Gamma$}
\label{subsec:O2RepsGenGamma}

An arbitrary unitary representation $V$ of $\OO(2)$ such that $V^{\OO(2)} = \{0\}$ is of the form
\begin{equation}
\label{eq:genO2Rep}
    V   =   \left(\bigoplus\limits_{i=1}^n \tau_{\alpha_i}\right)\oplus d\det
\end{equation}
where each $\alpha_i > 0$, $d\geq 0$, and $\dim V = 2n + d$. Note that
the underlying action of $\SO(2)\simeq\Sp^1$ has weight vector $(\alpha_1,-\alpha_1,\ldots,\alpha_n,-\alpha_n, \overset{d}{\overbrace{0,\ldots,0}})$.

Let $(x_{i,1},x_{i,2})$ be complex coordinates for $\tau_{\alpha_i}$ and let $y_j$ be the complex coordinate for the $j$th $\det$ factor.
Then the complex coordinates for $V$ are given by $(x_{1,1},x_{1,2},x_{2,1},\ldots,x_{n,1},x_{n,2},y_1,\ldots,y_d)\in\C^{2n+d}$.
It will be convenient to partition $V$
into subsets of orbit types as follows:
\begin{itemize}
\item   The origin, which has isotropy $\OO(2)$.
\item   For each $i\in\underline{n}$, the set $\boldsymbol{X}_i$ of points $\boldsymbol{v}\in V$ such that exactly one of $x_{i,1}$ and $x_{i,2}$ are
        nonzero and all other coordinates are zero.
        These points have isotropy group $R(\alpha_i)$ isomorphic to $\Z/\alpha_i\Z$.
\item   For each nonempty $I\subseteq\underline{n}$, the set $\boldsymbol{X}_I^\ast$ of points $\boldsymbol{v}\in V$ with  $x_{i,1}, x_{i,2} \neq 0$
        for $i\in I$ and all other coordinates zero such that $|x_{i,1}| = |x_{i,2}|$ for all $i$ and there is a common solution $\theta\in[0,2\pi)$ to
        $e^{\sqrt{-1}\,\alpha_i\theta} x_{i,2} = x_{i,1}$ for $i\in I$.
        It follows that there are $\gcd\{\alpha_i : i\in I\}$ such solutions so that these points have isotropy isomorphic to $D_{2\gcd\{\alpha_i : i\in I\}}$.
\item   For each $I\subseteq\underline{n}$ with at least two elements, the set $\boldsymbol{X}_I^\times$ of points $\boldsymbol{v}\in V$ such that
        either all of the $x_{i,1}\neq 0$ for $i\in I$ or all of the $x_{i,2}\neq 0$ for $i\in I$, with all other coordinates zero.
        These points have isotropy isomorphic to $\Z/\gcd\{\alpha_i:i\in I\}\Z$.
\item   For each nonempty $I\subseteq\underline{n}$, the set $\boldsymbol{X}_I$ of points $\boldsymbol{v}\in V$ such that at least one of $x_{i,1}$
        and $x_{i,2}$ are nonzero for each $i\in I$; $x_{i,1} = x_{i,2} = 0$ for $i\notin I$; $y_j = 0$ for all $j$;
        $\boldsymbol{v}\notin \boldsymbol{X}_I^\ast$; $\boldsymbol{v}\notin\boldsymbol{X}_I^\times$; and,
        if $I$ has one element $i$, both $x_{i,1}, x_{i,2}\neq 0$ so that $\boldsymbol{v}\notin \boldsymbol{X}_i$.
        These points have isotropy isomorphic to $\Z/\gcd\{\alpha_i:i\in I\}\Z$.
\item   For each nonempty $J\subseteq\underline{d}$, the set $\boldsymbol{Y}_J$ of points $\boldsymbol{v}\in V$ such that $y_j\neq 0$ for $j\in J$
        with all other coordinates zero.
        These points have isotropy $\SO(2)$.
\item   For each nonempty $I\subseteq\underline{n}$ and nonempty $J\subseteq\underline{d}$, the set $\boldsymbol{X}_{I,J}$ of points $\boldsymbol{v}\in V$
        such that at least one of $x_{i,1}$ and $x_{i,2}$ are nonzero for each $i\in I$, $y_j\neq 0$ for $j\in J$, and all other coordinates are zero.
        These points have isotropy isomorphic to $\Z/\gcd\{\alpha_i:i\in I\}\Z$.
\end{itemize}
Note that if $x_{i,1}\neq 0$ and $x_{i,2} = 0$ for a point $\boldsymbol{v}\in V$, then the action of $\OO(2)\smallsetminus \SO(2)$ maps
$\boldsymbol{v}$ to a point satisfying $x_{i,1}= 0$ and $x_{i,2}\neq 0$.
Otherwise, the $\OO(2)$-action preserves the zero and nonzero coordinates of $\boldsymbol{v}$.
If $\boldsymbol{v}\in\boldsymbol{X}_I^\ast$ for some
$I\subseteq\underline{n}$, then the action of $\OO(2)$ preserves the number of common solutions $\theta\in[0,2\pi)$ to
$e^{\sqrt{-1}\,\alpha_i\theta} x_{i,2} = x_{i,1}$ for $i\in I$.
Therefore, each of the sets $\boldsymbol{X}_i$, $\boldsymbol{X}_I^\ast$, $\boldsymbol{X}_I^\times$, $\boldsymbol{X}_I$, $\boldsymbol{Y}_J$, and
$\boldsymbol{X}_{I,J}$ defined above are $\OO(2)$-invariant. Note that points in $\boldsymbol{X}_I^\ast$, $\boldsymbol{X}_I^\times$, and
$\boldsymbol{X}_I$ have at least two nonzero coordinates.

\begin{theorem}
\label{thrm:GamECO2Rep}
Let $\rho\colon\OO(2)\to\U(V)$ be a unitary $\OO(2)$-representation with $V^{\OO(2)} = \{0\}$ of the form \eqref{eq:genO2Rep}
and let $\Gamma$ be a finitely presented discrete group. Then
\begin{equation}
\label{eq:GamECO2Rep}
    \chi_\Gamma\big(\OO(2)\ltimes V\big)
    =   \chi\big(\OO(2)\backslash\HOM(\Gamma,\OO(2))\big)
        -\sum\limits_{i=1}^n \chi\big(\HOM(\Gamma,\Z/\alpha_i\Z)\big).
\end{equation}
\end{theorem}
\begin{proof}
The orbit space of the origin is a single point with isotropy $\OO(2)$, yielding the term $\chi\big(\OO(2)\backslash\HOM(\Gamma,\OO(2))\big)$.
For $i\in\underline{n}$, the set $\OO(2)\backslash \boldsymbol{X}_i$ is homeomorphic to an open interval; specifically,
each point in $\boldsymbol{X}_i$ contains a unique point in its orbit such that $x_{i,1}$ is positive real and $x_{i,2} = 0$.
Therefore, each $\boldsymbol{X}_i$ contributes $-\chi\big(\HOM(\Gamma,\Z/\alpha_i\Z)\big)$ to $\chi_\Gamma(\OO(2)\ltimes V)$.

For a set $\boldsymbol{X}_I^\times$ corresponding to $I\subseteq\underline{n}$ with $|I|\geq 2$ elements such that the $\alpha_i$
with $i\in I$ all coincide,
$\boldsymbol{X}_I^\times$ is homeomorphic to $(\C\smallsetminus\{0\})^{|I|}\sqcup(\C\smallsetminus\{0\})^{|I|}$, and
$\OO(2)\backslash\boldsymbol{X}_I^\times$ is homeomorphic to $(0,1)\times(\C\smallsetminus\{0\})^{|I|-1}$. To see this,
note that the orbit of each $\boldsymbol{v}\in\boldsymbol{X}_I^\times$ contains a unique point such that $x_{i,1}\neq 0$ for each
$i\in I$ and, if $i_1$ is the smallest element of $I$, then $x_{i_1,1}$ is positive real; for $i\neq i_1$, the $x_{i,1}$ can take
any value in $\C\smallsetminus\{0\}$. As $|I| - 1 \geq 1$, it follows that
$\chi\big(\OO(2)\backslash\boldsymbol{X}_I^\times\big) = \chi\big((0,1)\times(\C\smallsetminus\{0\})^{|I|-1}\big) = 0$.

For any set $Z$ of the form $\boldsymbol{X}_I^\ast$, $\boldsymbol{X}_I$, $\boldsymbol{Y}_J$, or $\boldsymbol{X}_{I,J}$,
or of the form $\boldsymbol{X}_I^\times$ where the $\alpha_i$ for $i\in I$ do not all coincide, the circle action by
scalar multiplication on $V$ preserves $Z$ and has finite intersection with $\OO(2)$ in $\HOMEO(Z)$. Hence the corresponding contribution to
$\chi_\Gamma(\OO(2)\ltimes V)$ is zero by Corollary~\ref{cor:SimpCircleActionDiag}, completing the proof.
\end{proof}

\begin{remark}
\label{rem:GamECO2RepTrivial}
The hypothesis in Theorem~\ref{thrm:GamECO2Rep} that $V^{\OO(2)} = \{0\}$ can easily be removed by noting that, for an arbitrary
unitary $\OO(2)$-representation $W$, $W = W^{\OO(2)} \times V$ with $V$ as in \eqref{eq:genO2Rep}, where $W^{\OO(2)}$ is a complex
vector space so that $\chi\big(W^{\OO(2)}\big) = 1$. Hence, as in Remark~\ref{rem:GamECS1RepReal},
$\chi_\Gamma\big(\OO(2)\ltimes V\big) = \chi_\Gamma\big(\OO(2)\ltimes W\big)$.
\end{remark}

Continuing to let $V$ denote a unitary $\OO(2)$-representation such that $V^{\OO(2)} = \{0\}$ of the form \eqref{eq:genO2Rep},
let $V_\R \subset V$ denote the set of real points in $V$. That is,
\[
    V_\R    =   \{ (x_{1,1},x_{1,2},x_{2,1},\ldots,x_{n,1},x_{n,2},y_1,\ldots,y_d) \in V :
                    \overline{x_{i,1}} = x_{i,2} \;\forall i \text{ and } y_j\in\R \:\forall j\}.
\]
Then $V_\R$ is a real representation of $\OO(2)$. For example, the defining representation of $\OO(2)$ is given by $(\tau_1)_\R$.

\begin{theorem}
\label{thrm:GamECO2RealRep}
Let $\rho\colon\OO(2)\to\U(V)$ be a unitary $\OO(2)$-representation with $V^{\OO(2)} = \{0\}$ of the form \eqref{eq:genO2Rep}
and let $\Gamma$ be a finitely presented discrete group. Then
\begin{equation}
\label{eq:GamECO2RealRep}
\begin{split}
    \chi_\Gamma\big(\OO(2)\ltimes V_\R\big)
    &=          \chi\big(\OO(2)\backslash\HOM(\Gamma,\OO(2))\big)
                + \frac{(-1)^d - 1}{2}\chi\big(\HOM(\Gamma,\Sp^1)\big)
    \\&\quad    - \sum\limits_{\emptyset\neq I\subseteq\underline{n}}
                (-2)^{|I|-1}\chi\big(D_{2\gcd\{\alpha_I\}}\backslash\HOM(\Gamma,D_{2\gcd\{\alpha_I\}})\big)
    \\&\quad    - \sum\limits_{I = \{i_1,i_2\}\subseteq\underline{n}}
                \chi\big(D_{2\gcd\{\alpha_I\}}\backslash\HOM(\Gamma,D_{2\gcd\{\alpha_I\}})\big)
    \\&\quad    + \frac{1- (-1)^d}{2} \sum\limits_{i=1}^n \chi\big(\HOM(\Gamma,\Z/\alpha_i\Z)\big),
\end{split}
\end{equation}
where we use the shorthand $\{ \alpha_I\}$ to denote $\{\alpha_i:i\in I\}$.
\end{theorem}
\begin{proof}
The orbit space of the origin is again a point with isotropy $\OO(2)$, yielding the term $\chi\big(\OO(2)\backslash\HOM(\Gamma,\OO(2))\big)$.
The sets $\boldsymbol{X}_i$ and $\boldsymbol{X}_I^\times$ do not intersect $V_\R$, so we consider the sets $\boldsymbol{X}_I^\ast$
for nonempty $I\subseteq\underline{n}$ with $|I|$ elements.
Let $I = \{i_1,\ldots,i_k\}$ and then the $\OO(2)$-orbit of $(x_{i_1,1},x_{i_1,2})\in(\tau_{i_1})_\R$ contains a unique element with $x_{i_1,1} = x_{i_1,2}$
positive real, which has isotropy $D_{2\alpha_{i_1}}$. In order to be an element of $\boldsymbol{X}_I^\ast$,
$(x_{i_2,1},x_{i_2,2})\in(\tau_{i_2})_\R$ must be contained in one of $\alpha_{i_1}/\gcd(\alpha_{i_1},\alpha_{i_2})$ lines through the origin
(excluding the origin),
and the action of $D_{2\alpha_{i_1}}$ identifies these lines. Hence the $D_{2\alpha_{i_1}}$-orbit of $(x_{i_2,1},x_{i_2,2})$ contains a unique
element with $x_{i_2,1}$ on one of two rays (excluding the origin), the positive real axis and the set of points with argument
$\pi\gcd(\alpha_{i_1},\alpha_{i_2})/\alpha_{i_1}$; such points
$(x_{i_1,1},x_{i_1,2},x_{i_2,1},x_{i_2,2})\in(\tau_{i_1} \oplus \tau_{i_2})_\R$ have isotropy $D_{2\gcd(\alpha_{i_1},\alpha_{i_2})}$.
In the same way, the $D_{2\gcd(\alpha_{i_1},\alpha_{i_2})}$-orbit of $(x_{i_3,1},x_{i_3,2})$ contains a unique element on one of two rays
excluding the origin. Continuing by induction, the orbit space of $\boldsymbol{X}_I^\ast$ is homeomorphic to
$(0,1)\times \big((0,1)\sqcup(0,1)\big)^{|I|-1}$. It therefore
has Euler characteristic $-(-2)^{|I|-1}$ and isotropy $D_{2\gcd\{\alpha_i:i\in I\}}$, yielding the first sum in
Equation~\eqref{eq:GamECO2RealRep}.

The set $\boldsymbol{X}_I$ only intersects $V_\R$ if $I$ has at least two elements, as any real point $(x_{i,1},x_{i,2})$ satisfies
$|x_{i,1}| = |x_{i,2}|$ and hence is an element of $\boldsymbol{X}_I^\ast$. First assume $I = \{i_1,i_2\}$ has two elements, and then as
in the case of $\boldsymbol{X}_I^\ast$, there is a unique point in the $\OO(2)$-orbit of $(x_{i_1,1},x_{i_1,2})\in(\tau_{i_1})_\R$ such that
$x_{i_1,1} = x_{i_1,2}$ is positive real; this point has isotropy $D_{2\alpha_{i_1}}$. Then $x_{i_2,1}$ must lie in the complement of the
$\alpha_{i_1}/\gcd(\alpha_{i_1},\alpha_{i_2})$ lines through the origin fixed by a reflection in $D_{2\alpha_{i_1}}$. This separates $\C$
into $2\alpha_{i_1}/\gcd(\alpha_{i_1},\alpha_{i_2})$ open $2$-cells, and there is a unique point in the $D_{2\alpha_{i_1}}$-orbit of
$(x_{i_2,1},x_{i_2,2})\in(\tau_{i_2})_\R$ contained in the $2$-cell bounded by the positive real axis and the ray of points with
argument $\pi\gcd(\alpha_{i_1},\alpha_{i_2})/\alpha_{i_1}$. It follows that the $\OO(2)$-orbit space of $\boldsymbol{X}_I$ is homeomorphic to
$(0,1)\times\C$ and hence $\chi\big(\OO(2)\backslash\boldsymbol{X}_I\big) = -1$ with isotropy $D_{2\gcd(\alpha_{i_1},\alpha_{i_2})}$.
Now suppose $I$ has at least three elements, let $i_1$ be the smallest, and then there must be at least one other element $i_2$ such that
$(x_{i_2,1},x_{i_2,2})\in(\tau_{i_2})_\R$ is not fixed by any of the reflections in $D_{2\alpha_{i_1}}$ (otherwise, the point is in
$\boldsymbol{X}_I^\ast$). Choose a third element $i_3\in I$, and then $x_{i_3,1}$ is unconstrained in $\C\smallsetminus\{0\}$.
Arguing as in the case $I$ has two elements, there is a unique point in the $\OO(2)$-orbit of
$(x_{i_1,1},x_{i_1,2},x_{i_2,1},x_{i_2,2},x_{i_3,1},x_{i_3,2})\in(\tau_{i_1}\oplus\tau_{i_2}\oplus\tau_{i_3})_\R$ such that
$x_{i_1,1}$ is positive real, $x_{i_2,1}$ is contained in a single open $2$-cell, and $x_{i_3,1}$ is an element of
$(\Z/\gcd(\alpha_{i_1},\alpha_{i_2},\alpha_{i_2})\Z)\backslash(\C\smallsetminus\{0\})$ with the cyclic group acting as rotations.
It follows that the $\OO(2)$-orbit space of $\boldsymbol{X}_I$ is homeomorphic to
$(0,1)\times\C\times(\C\smallsetminus\{0\})\times\cdots$ and hence has vanishing Euler characteristic. Note that if $I$ has more than
three elements, the remaining factors contributed by each will each be homeomorphic $\C\smallsetminus\{0\}$ by the same argument as in the case
of $x_{i_3}$. It follows that the $\boldsymbol{X}_I$ contribute the second sum in Equation~\eqref{eq:GamECO2RealRep}.

Now let $J\subseteq\underline{d}$ be nonempty with $|J|$ elements, and then $\boldsymbol{Y}_J$ is homeomorphic to $(\R\smallsetminus\{0\})^{|J|}$.
The action of $\OO(2)$ is by scalar multiplication by $-1$, so the orbit space $\OO(2)\backslash\boldsymbol{Y}_J$ is homeomorphic to
$\R\times\big(\R\smallsetminus\{0\}\big)^{|J|-1}$ and has Euler characteristic $-(-2)^{|J|-1}$. These points have abelian isotropy $\Sp^1$ so that
the $\boldsymbol{Y}_J$ contribute
\begin{align*}
    - \sum\limits_{\emptyset\neq J\subseteq\underline{d}} (-2)^{|J|-1}\chi\big(\HOM(\Gamma,\Sp^1)\big)
    &=      - \chi\big(\HOM(\Gamma,\Sp^1)\big) \sum\limits_{|J|=1}^d {d \choose |J|} (-2)^{|J|-1}
    \\&=    \frac{(-1)^d - 1}{2}\chi\big(\HOM(\Gamma,\Sp^1)\big)
\end{align*}
to $\chi_\Gamma\big(\OO(2)\ltimes V_\R\big)$.

It remains only to consider the $\boldsymbol{X}_{I,J}$ for nonempty sets $I\subset\underline{n}$ and $J\subset\underline{d}$.
First assume $I = \{i\}$ is a singleton and $J$ is arbitrary with $|J|>0$ elements, and let $j_1$ be the smallest element of $J$.
The $\OO(2)$-orbit of $(x_{i,1},x_{i,2},y_{j_1})$ contains a unique point with $x_{i,1}$ positive real and $y_{j_1}$ positive,
and the isotropy group $\Z/\alpha_i\Z$ of this point acts trivially on the remaining coordinates $y_j$ for $j\in J$. It follows that
the orbit space $\OO(2)\backslash\boldsymbol{X}_{I,J}$ of $\boldsymbol{X}_{I,J}$ is homeomorphic to
$\R^2\times\big(\R\smallsetminus\{0\}\big)^{|J|-1}$ and hence has Euler characteristic $(-2)^{|J|-1}$, contributing
$(-2)^{|J|-1}\chi\big(\HOM(\Gamma,\Z/\alpha_i\Z)\big)$ to $\chi_\Gamma\big(\OO(2)\ltimes V_\R\big)$.
If $I = \{i_1,i_2,\ldots\}$ has at least two elements, then choosing an $\OO(2)$-orbit representative such that
$x_{i_1,1}$ is positive real and $y_{j_1} > 0$, $x_{i_2,1}$ can take any nonzero value in $\C\smallsetminus\{0\}$, and the
$\Z/\alpha_{i_1}\Z$-action on $\C\smallsetminus\{0\}$ results in a $(\Z/\alpha_{i_1}\Z)\backslash(\C\smallsetminus\{0\})$
factor in the orbit space homeomorphic to $\C\smallsetminus\{0\}$.
It follows that $\OO(2)\backslash\boldsymbol{X}_{I,J}$ has vanishing Euler characteristic. Hence, the contribution to
to $\chi_\Gamma\big(\OO(2)\ltimes V_\R\big)$ of the sets $\boldsymbol{X}_{I,J}$ is given by
\begingroup
\allowdisplaybreaks
\begin{align*}
    \sum\limits_{i=1}^n \sum\limits_{\emptyset\neq J\subseteq\underline{d}} (-2)^{|J|-1} & \chi\big(\HOM(\Gamma,\Z/\alpha_i\Z)\big)
    \\&=    \sum\limits_{i=1}^n \chi\big(\HOM(\Gamma,\Z/\alpha_i\Z)\big) \sum\limits_{|J|=1}^d {d \choose |J|} (-2)^{|J|-1}
    \\&=    \sum\limits_{i=1}^n \frac{1- (-1)^d}{2} \chi\big(\HOM(\Gamma,\Z/\alpha_i\Z)\big).
    \qedhere
\end{align*}
\endgroup
\end{proof}


\subsection{Computations for $\OO(2)$-representations with free and free abelian $\Gamma$}
\label{subsec:O2RepsFreeGroups}

To specialize the computations of Section~\ref{subsec:O2RepsGenGamma} to the groups $\Gamma = \Z^\ell$ and $\F_\ell$, we need only compute
$\chi\big(H\backslash\HOM(\Gamma,H)\big)$ for the isotropy groups $H$ that appear in Equations~\eqref{eq:GamECO2Rep} and \eqref{eq:GamECO2RealRep}.
For abelian $H$, this computation is trivial and was considered in Section~\ref{sec:CircleReps}; we have
$\chi\big(\Sp^1\backslash\HOM(\Gamma,\Sp^1)\big) = \chi\big(\HOM(\Gamma,\Sp^1)\big) = 0$ and
$\chi\big((\Z/\alpha\Z)\backslash\HOM(\Gamma,(\Z/\alpha\Z))\big) = \chi\big(\HOM(\Gamma,(\Z/\alpha\Z))\big) = \alpha^\ell$ for both
$\Gamma = \Z^\ell$ and $\Gamma = \F_\ell$. We now consider the remaining cases.

By choosing a generating set for $\F_\ell$, $\HOM(\F_\ell,H)$ can be identified with the set of ordered $\ell$-tuples of elements of $H$,
and the action of $\F_\ell$ on $\HOM(\F_\ell,H)$ corresponds to simultaneous conjugation of an $\ell$-tuple.
The partition of the set of $\ell$-tuples of elements of $\OO(2)$ into $\OO(2)$-orbit types under simultaneous conjugation is as follows.
Note that the center of $\OO(2)$ is given by $\{r_0,r_\pi\}$, the centralizer of any other rotation is $\SO(2)$, and the centralizer of a reflection
$s_\theta$ is $\{r_0,r_\pi,s_\theta,r_\pi s_\theta = s_{\theta+\pi}\}$; all $s_\theta$ are conjugate, and $r_\theta$ is conjugate only to $r_{-\theta}$.
\begin{enumerate}
\item[(i)]      $\ell$-tuples $(g_1,\ldots,g_\ell)\in\OO(2)$ such that each $g_i\in\{r_0,r_\pi\}$.
                These $\ell$-tuples are fixed by $\OO(2)$.
\item[(ii)]     $\ell$-tuples $(g_1,\ldots,g_\ell)\in\OO(2)$ such that each $g_i\in\SO(2)$ and at least one $g_i\notin\{r_0,r_\pi\}$.
                The isotropy group of these $\ell$-tuples is $\SO(2)$.
\item[(iii)]    $\ell$-tuples $(g_1,\ldots,g_\ell)\in\OO(2)$ in which a single reflection $s_\theta$ occurs at least once and all $g_i$ commute
                with $s_\theta$.
                The isotropy group of any such $\ell$-tuple is $\{r_0,r_\pi,s_\theta,s_{\theta+\pi}\}$.
\item[(iv)]    $\ell$-tuples $(g_1,\ldots,g_\ell)\in\OO(2)$ fixed only by the center $\{r_0,r_\pi\}$ of $\OO(2)$. These $\ell$-tuples
                contain at least one reflection $s_\theta$ and at least one $g_i$ that does not commute with $s_\theta$.
                These $\ell$-tuples have isotropy $\{r_0,r_\pi\}$.
\end{enumerate}
The $\ell$-tuples of types (i), (ii), and (iii) are commuting $\ell$-tuples and hence represent elements of $\HOM(\Z^\ell,\OO(2))$, while the
$\ell$-tuples of type (iv) do not commute.

\begin{lemma}
\label{lem:ZlFlO2}
For any $\ell\geq 1$,
\begin{equation}
\label{eq:ZlO2}
    \chi\big(\OO(2)\backslash\HOM(\Z^\ell,\OO(2))\big)
        =   2^{2\ell-1},
\end{equation}
and for $\ell\geq 2$, we have
\begin{equation}
\label{eq:FlO2}
    \chi\big(\OO(2)\backslash\HOM(\F_\ell,\OO(2))\big)
        =   2^{\ell-2}(2^\ell + 1).
\end{equation}
\end{lemma}
\begin{proof}
We consider the orbits of $\ell$-tuples of each type listed above and compute
$\chi\big(\OO(2)\backslash\HOM(\Gamma,\OO(2))\big)$ for the $\Gamma$ under consideration.

As the center of $\OO(2)$ has two elements, there are $2^\ell$ $\ell$-tuples of type (i), and the $\OO(2)$-orbit of each is a singleton.
The space of $\ell$-tuples of types (i) and (ii) combined form the space $\SO(2)^\ell$, and those of type (i) yield $2^\ell$ discrete
points in $\SO(2)^\ell$. Hence, the $\ell$-tuples of type (ii) consist of the complement of these points, a space with Euler characteristic
$-2^\ell$. Each type (ii) $\ell$-tuple is centralized by $\SO(2)$ and hence has an orbit with two elements; it follows that the space of
type (ii) tuples forms a $2$-to-$1$ cover of its orbit space, which therefore has Euler characteristic $-2^{\ell-1}$. The space of orbits of
$\ell$-tuples of type (iii) is discrete; although there are infinitely many choices for $s_\theta$, they are all conjugate, so we can for instance
choose the unique representative of each orbit such that the first reflection to occur in the $\ell$-tuple is $s_0$.
To count the resulting orbits of $\ell$-tuples, let $r$ be the first position for which $g_r$ is a reflection. There are $2$ choices for each
$g_i$ with $i < \ell$ and $4$ choices for each $g_i$ with $i > r$, yielding
$\sum_{r=1}^\ell 2^{r-1} 4^{\ell-r} = 2^{\ell-1}(2^\ell - 1)$ orbits. Because $\HOM(\Z^\ell,\OO(2))$ is identified with the space of commuting
$\ell$-tuples, i.e., those of types (i), (ii), and (iii), we have
\[
    \chi\big(\OO(2)\backslash\HOM(\Z^\ell,\OO(2))\big)
    =
    2^\ell
    -
    2^{\ell-1}
    +
    2^{\ell-1}(2^\ell - 1)
    =
    2^{2\ell-1},
\]
establishing Equation~\eqref{eq:ZlO2}.

If $\ell = 1$, then there are no $\ell$-tuples of type (iv), as any singleton containing a reflection is type (iii). We claim by induction
on $\ell$ that if $\ell\geq 2$, then the space of $\OO(2)$-orbits of $\ell$-tuples of type (iv) has Euler characteristic
$-2^{\ell-2}(2^\ell - 1)$. First consider $\ell = 2$. A type (iv) tuple must contain a reflection, and there is a unique element of each orbit
such that the first reflection that occurs is $s_0$ and the other element is either a rotation $r_\theta$ or a reflection $s_\theta$
with $\theta\in(0,\pi)$. The orbit space is therefore three intervals: the tuples of the form $(s_0, s_\theta)$, those of the form
$(s_0, r_\theta)$, and those of the form $(r_\theta, s_0)$, so that the Euler characteristic is $-3 = -2^{\ell-2}(2^\ell - 1)$.

Now, let $X$ be the space of $\ell$-tuples of type (iv) for some $\ell\geq 2$, and assume $\chi(\OO(2)\backslash X) = -2^{\ell-2}(2^\ell - 1)$.
If $(g_1,\ldots,g_{\ell+1})$ is a type (iv) $(\ell+1)$-tuple, then $(g_1,\ldots,g_\ell)$ must be type (ii), (iii), or (iv); we first consider
the case that $(g_1,\ldots,g_\ell)$ is type (iv). Choosing a representative of each orbit in $\OO(2)\backslash X$ and concatenating any
$g_{\ell+1}\in\OO(2)$ to this representative yields a representative of the orbit of a type (iv) $(\ell+1)$-tuple. Therefore, the set of
orbits of type (iv) $(\ell+1)$-tuples such that $(g_1,\ldots,g_\ell)$ is as well type (iv) is homeomorphic to $(\OO(2)\backslash X)\times\OO(2)$
and has zero Euler characteristic.

As seen above, the space of orbits of type (ii) $\ell$-tuples has Euler characteristic $-2^{\ell-1}$, and each is centralized by $\SO(2)$.
Choose a representative $(g_1,\ldots,g_\ell)$ of each orbit, and then we may concatenate any reflection $g_{\ell+1} = s_\theta$ to yield a
type (iv) $(\ell+1)$-tuple. Conjugating by $\SO(2)$, there is a unique representative of the resulting $(\ell+1)$-tuple such that
$g_{\ell+1} = s_0$. It follows that the space of orbits of type (iv) $(\ell+1)$-tuples whose first $\ell$ coordinates are type (ii) has
Euler characteristic $-2^{\ell-1}$.

Similarly, the space consisting of orbits of type (iii) $\ell$-tuples has Euler characteristic $2^{\ell-1}(2^\ell - 1)$. We can choose a unique
representative of each orbit as above such that the first reflection to occur is $s_0$, and the centralizer of this $\ell$-tuple is
$\{r_0,r_\pi,s_0,s_\pi\}$. We may form a type (iv) $(\ell+1)$-tuple by concatenating any element of $\OO(2)$ that is not in this centralizer,
and then we may conjugate by $s_0$ to choose a unique element of the orbit of the resulting $(\ell+1)$-tuple such that
$g_{\ell+1} = s_\theta$ or $r_\theta$ with $\theta\in(0,\pi)$. It follows that the Euler characteristic of the space of $(\ell+1)$-tuples
formed in this way is $-2^\ell(2^\ell - 1)$. Combining the above computations, the space of orbits of type (iv) $(\ell+1)$-tuples has
Euler characteristic
\[
    -2^\ell(2^\ell - 1) - 2^{\ell-1} = - 2^{\ell-1}(2^{\ell+1}-1),
\]
completing the induction.

With this, Equation~\eqref{eq:FlO2} is obtained by adding $-2^{\ell-2}(2^\ell - 1)$, the Euler characteristic of the space of
orbits of type (iv) $\ell$-tuples, to Equation~\eqref{eq:ZlO2}.
\end{proof}

Combining Lemma~\ref{lem:ZlFlO2} with Theorem~\ref{thrm:GamECO2Rep} yields the following.

\begin{corollary}
\label{cor:GamECO2RepFlZl}
Let $\rho\colon\OO(2)\to\U(V)$ be a unitary $\OO(2)$-representation of the form \eqref{eq:genO2Rep} such that $V^{\OO(2)} = \{0\}$.
For any $\ell\geq 1$,
\[
    \chi_{\Z^\ell}\big(\OO(2)\ltimes V\big)
    =   2^{2\ell-1} - \sum\limits_{i=1}^n \alpha_i^\ell,
\]
and for $\ell \geq 2$,
\[
    \chi_{\F_\ell}\big(\OO(2)\ltimes V\big)
    =   2^{\ell-2}(2^\ell + 1) - \sum\limits_{i=1}^n \alpha_i^\ell.
\]
\end{corollary}

\begin{remark}
\label{rem:FreeECO2RepDependence}
As in Remark~\ref{rem:FreeECS1RepDependence}, observe that the $\chi_{\Z^\ell}\big(\OO(2)\ltimes V\big)$ for $\ell=1,\ldots,n$
determine the $\alpha_i$, and the same holds for the $\chi_{\F_\ell}\big(\OO(2)\ltimes V\big)$. Hence, among unitary $\OO(2)$-representations
of the same dimension, finitely many of the $\chi_{\Z^\ell}\big(\OO(2)\ltimes V\big)$ or $\chi_{\F_\ell}\big(\OO(2)\ltimes V\big)$
determine the representation.
\end{remark}

\begin{remark}
\label{rem:ExactSeq}
With $V$ as in Theorem~\ref{thrm:GamECO2Rep}, note that $\OO(2)\ltimes V$ fits into a short exact sequence of Lie groupoids
\[
    \mathbf{1}_V      \longrightarrow
    \Sp^1\ltimes V  \overset{\nu}{\longrightarrow}
    \OO(2)\ltimes V \overset{\kappa}{\longrightarrow}
    \Z/2\Z\ltimes (\Sp^1\backslash V) \longrightarrow
    \mathbf{1}_V
\]
where $\mathbf{1}_V$ denotes the base groupoid $V\rightrightarrows V$ over $V$ in which all arrows are units,
$\nu$ is the embedding of the connected component containing the units, and
$\kappa$ is the quotient map. The $\Sp^1$-action has weight vector
$(\alpha_1,-\alpha_1,\ldots,\alpha_n,-\alpha_n, \overset{d}{\overbrace{0,\ldots,0}})$
as noted after Equation~\eqref{eq:genO2Rep}. The nontrivial element of $\Z/2\Z$ acts
on $\Sp^1\backslash V$ by
\[
    \Sp^1(x_{1,1},x_{1,2},\ldots,x_{n,1},x_{n,2},y_1,\ldots,y_d)
        \mapsto \Sp^1(x_{1,2},x_{1,1},\ldots,x_{n,2},x_{n,1},-y_1,\ldots,-y_d).
\]
The action of $\Sp^1$ on $V$ by scalar multiplication fixes only the origin and commutes with the $\OO(2)$-action
so that by Corollary~\ref{cor:SimpTorusActionOSDGroupoid},
\[
    \chi_\Gamma\big(\Z/2\Z\ltimes (\Sp^1\backslash V)\big)    =   \chi\big(\HOM(\Gamma,\Z/2\Z)\big)
\]
for each finitely presented $\Gamma$. In particular, for each positive integer $\ell$, it follows that
$\chi_{\Z^\ell}\big(\Z/2\Z\ltimes (\Sp^1\backslash V)\big) =
\chi_{\F_\ell}\big(\Z/2\Z\ltimes (\Sp^1\backslash V)\big) = 2^\ell$.
We have $\chi_\Gamma(\mathbf{1}_V) = 1$ for all $\Gamma$, and by Corollary~\ref{cor:FreeECS1Rep},
\[
    \chi_{\Z^\ell}(\Sp^1\ltimes V)  =   \chi_{\F_\ell}(\Sp^1\ltimes V)  =   -2\sum\limits_{i=1}^n \alpha_i^\ell.
\]
In \cite[Theorem~5.13]{FarsiSeatonEC}, it was demonstrated that $\chi_\Gamma$ is multiplicative over short exact sequences
$1\to\mathbf{B}\to\G\to\mathbf{H}\to 1$ where $\mathbf{B}$ is a bundle of compact Lie groups, $\mathbf{H}$ is a translation
groupoid, and the isotropy groups in $\G$ are abelian; \cite[Example~14]{FarsiSeatonEC} then demonstrated that there is no
simple relationship between the $\chi_\Gamma$ when the assumption that the $\G_x^x$ are abelian was dropped. Comparing
the above $\chi_\Gamma$ to Corollary~\ref{cor:GamECO2RepFlZl} further indicates that multiplicativity of $\chi_\Gamma$
over short exact sequences only holds in specific circumstances. In particular, $\chi_{\F_\ell}$ and $\chi_{\Z^\ell}$
coincide for $\mathbf{1}_V$, $\Sp^1\ltimes V$, and $\Z/2\Z\ltimes (\Sp^1\backslash V)$ yet do not for
$\OO(2)\ltimes V$.
\end{remark}

We now consider $\chi\big(D_{2m}\backslash\HOM(\Z^\ell,D_{2m})\big)$.
The center of $D_{2m}$ is $\{r_0,r_\pi\}$ if $m$ is even and trivial if $m$ is odd. If $m$ is even, there are two conjugacy
classes of reflections; if $m$ is odd, all reflections are conjugate. For any $m$, the $D_{2m}$-conjugacy class of any rotation
$r_\theta$ with $\theta\neq 0,\pi$ is the pair $r_{\pm\theta}$.

\begin{lemma}
\label{lem:ZlFlDm}
Let $m$ be any positive integer and let $P(m) = 2$ if $m$ is even and $1$ if $m$ is odd. For any $\ell\geq 1$,
\begin{equation}
\label{eq:DmZlO2}
    \chi\big(D_{2m}\backslash\HOM(\Z^\ell,D_{2m})\big)
        =   \frac{m^\ell + P(m)^\ell(2^{\ell+1} - 1)}{2},
\end{equation}
and for $\ell\geq 2$, we have
\begin{equation}
\label{eq:DmFlO2}
    \chi\big(D_{2m}\backslash\HOM(\F_\ell,D_{2m})\big)
        =   \frac{(2P(m))^{\ell} + P(m) m^{\ell-1}(2^{\ell} - 1) + m^\ell }{2}.
\end{equation}
\end{lemma}
\begin{proof}
First assume $m$ is odd, in which case $r_\pi\notin D_{2m}$. Then there is only one type (i) $\ell$-tuple in $D_{2m}^\ell$, and its orbit is a singleton.
There are $m^\ell-1$ type (ii) $\ell$-tuples, each with an orbit of size $2$, yielding $(m^\ell - 1)/2$ orbits. There are $m(2^\ell - 1)$ type (iii)
$\ell$-tuples that contain a single reflection and the identity, each with an orbit of size $m$, yielding $2^\ell - 1$ orbits. Hence the orbits of
commuting $\ell$-tuples consist of
\[
    1   +   \frac{m^\ell - 1}{2}    +   (2^\ell - 1)
        =   \frac{m^\ell + 2^{\ell + 1} - 1}{2}
\]
points, yielding Equation~\eqref{eq:DmZlO2} in this case.
Subtracting the numbers of type (i), (ii), and (iii) $\ell$-tuples from the total number $(2m)^\ell$ of $\ell$-tuples yields
$(m^\ell - m)(2^\ell - 1)$ type (iv) $\ell$-tuples, each with an orbit of size $2m$, hence $(m^{\ell-1} - 1)(2^\ell - 1)/2$ orbits of type (iv)
$\ell$-tuples. Then $\chi\big(D_{2m}\backslash\HOM(\F_\ell,D_{2m})\big)$ for $m$ odd is given by
\[
    \frac{m^\ell + 2^{\ell + 1} - 1}{2} + \frac{(m^{\ell-1} - 1)(2^\ell - 1)}{2}
    =
    \frac{2^{\ell} + m^{\ell-1}(2^\ell - 1) + m^\ell}{2},
\]
corresponding to Equation~\eqref{eq:DmFlO2} in this case.

Now assume $m$ is even. There are $2^\ell$ type (i) $\ell$-tuples, each with trivial orbit, and $m^\ell - 2^\ell$ type (ii) $\ell$-tuples partitioned
into orbits of size $2$. For a given reflection $s_\theta$, there are $4^\ell - 2^\ell$ tuples consisting of $\{r_0, r_\pi, s_\theta, s_{-\theta}\}$
that are not of type (ii), and $m/2$ choices for $s_\theta$, yielding $m(4^\ell - 2^\ell)/2$ type (iii) $\ell$-tuples in orbits of size $m/2$.
Hence,
\[
    \chi\big(D_{2m}\backslash\HOM(\Z^\ell,D_{2m})\big)
    =
    2^\ell
    + \frac{m^\ell - 2^\ell}{2}
    + (4^\ell - 2^\ell)
    =
    \frac{m^\ell + 2^\ell(2^{\ell+1} - 1)}{2},
\]
completing the proof of Equation~\eqref{eq:DmZlO2}.

We again count the $\ell$-tuples of type (iv) by subtracting those of type (i), (ii), and (iii) from $(2m)^\ell$, yielding
\[
    (2m)^\ell - 2^\ell - (m^\ell - 2^\ell) - \frac{m(4^\ell - 2^\ell)}{2}
    =
    \frac{(2^{\ell} - 1)(2m^\ell - m 2^\ell)}{2}
\]
$\ell$-tuples of type (iv) in orbits of size $m$. Adding the resulting number of orbits to $\chi\big(D_{2m}\backslash\HOM(\Z^\ell,D_{2m})\big)$
yields $\chi\big(D_{2m}\backslash\HOM(\F_\ell,D_{2m})\big)$ and completes the proof of Equation~\eqref{eq:DmFlO2}.
\end{proof}

Combining Lemmas~\ref{lem:ZlFlO2} and \ref{lem:ZlFlDm} with Theorem~\ref{thrm:GamECO2RealRep} yields the following.

\begin{corollary}
\label{cor:GamECO2RealRepFlZl}
Let $\rho\colon\OO(2)\to\U(V)$ be a unitary $\OO(2)$-representation of the form \eqref{eq:genO2Rep} with $V^{\OO(2)} = \{0\}$
and let $V_\R$ denote the set of real points in $V$. Let $P(m) = 2$ if $m$ is even and $1$ if $m$ is odd. For any $\ell\geq 1$,
using the shorthand $\{\alpha_I\}$ to denote $\{\alpha_i:i\in I\}$,
\[
\begin{split}
    \chi_{\Z^\ell} & \big(\OO(2)\ltimes V_\R\big)
        = 2^{2\ell-1}
        \\& + \sum\limits_{\emptyset\neq I\subseteq\underline{n}}
            (-2)^{|I|-2}
            \big(\gcd\{\alpha_I\}^\ell + P(\gcd\{\alpha_I\})^\ell(2^{\ell+1} - 1)\big)
    \\&\quad
        - \frac{1}{2}\sum\limits_{I = \{i_1,i_2\}\subseteq\underline{n}}
            \big(\gcd\{\alpha_I\}^\ell + P(\gcd\{\alpha_I\})^\ell(2^{\ell+1} - 1)\big)
        + \frac{1- (-1)^d}{2}\sum\limits_{i=1}^n \alpha_i^\ell,
\end{split}
\]
and for $\ell \geq 2$,
\[
\begin{split}
    &\chi_{\F_\ell} \big(\OO(2)\ltimes V_\R\big)
    =  2^{\ell-2}(2^\ell + 1) + \frac{1- (-1)^d}{2} \sum\limits_{i=1}^n \alpha_i
    \\&
        + \sum\limits_{\emptyset\neq I\subseteq\underline{n}}
            (-2)^{|I|-2}
            \big((2P(\gcd\{\alpha_I\}))^{\ell} + P(\gcd\{\alpha_I\}) \gcd\{\alpha_I\}^{\ell-1}(2^{\ell} - 1) + \gcd\{\alpha_I\}^\ell \big)
    \\&\quad
        - \frac{1}{2} \sum\limits_{I = \{i_1,i_2\}\subseteq\underline{n}}
            \big( (2P(\gcd\{\alpha_I\}))^{\ell} + P(\gcd\{\alpha_I\}) \gcd\{\alpha_I\}^{\ell-1}(2^{\ell} - 1)
            + \gcd\{\alpha_I\}^\ell \big).
\end{split}
\]
\end{corollary}


\section{Computations for linear symplectic quotients}
\label{sec:SympQuot}

In this section, we compute $\chi_\Gamma(\G)$ in the case that $\G$ is the translation groupoid representing the real linear symplectic quotient
associated to a unitary representation of an arbitrary compact Lie group $G$. In contrast to many of the other cases considered in this paper,
we demonstrate that $\chi_\Gamma(\G)$ only depends on the group $G$ and not on the representation.

Let $G$ be a compact Lie group and $\rho\colon G\to\GL(V)$ a finite-dimensional linear $G$-representation.
Equip $V$ with a $G$-invariant hermitian inner product, and then the imaginary part of the hermitian product is a symplectic form
on the underlying real vector space of $V$. The $G$-action on $V$ is Hamiltonian, admitting a unique homogeneous quadratic
moment map $\mu\colon V\to\mathfrak{g}^\ast$ where $\mathfrak{g}$ denotes the Lie algebra of $G$ and $\mathfrak{g}^\ast$ its dual.
Specifically, identifying $\mathfrak{g}^\ast$ with $\R^\ell$ by choosing a basis, the component functions of $\mu$ have bi-degree $(1,1)$,
i.e., are sums of terms of the form $x_i\overline{x_j}$ where $(x_1,\ldots,x_n)$ are complex coordinates for $V$.
The set $\mu^{-1}(0)$, called the \emph{shell}, is $G$-invariant and real algebraic, hence definable.
The \emph{real linear symplectic quotient at level $0$} is the quotient $G\backslash\mu^{-1}(0)$.
Note that if $G$ is finite, the moment map is zero so that $\mu^{-1}(0) = V$ and the symplectic
quotient is the usual quotient $G\backslash V$.
See \cite{ArmsGotayJennings,SjamaarLerman,HerbigSchwarzSeaton} for more information.

\begin{theorem}
\label{thrm:SympQuot}
Let $G$ be a compact Lie group, $\rho\colon G\to\GL(V)$ a finite-dimensional linear $G$-representation, and $\Gamma$
a finitely presented discrete group. Equip $V$ with a $G$-invariant hermitian inner product and
let $\mu\colon V\to\mathfrak{g}^\ast$ denote the homogeneous quadratic moment map so that the real symplectic quotient at
level $0$ is the orbit space $\vert G\ltimes\mu^{-1}(0)\vert$ of the groupoid $G\ltimes\mu^{-1}(0)$. Then
\begin{equation}
\label{eq:SympQuot}
    \chi_\Gamma\big( G\ltimes\mu^{-1}(0)\big)
        =   \chi\big(G\backslash\HOM(\Gamma,G)\big)
        =   \chi_{\Gamma}(G \ltimes \{pt\}).
\end{equation}
\end{theorem}
\begin{proof}
Let
\[
    \mu^{-1}(0) =   X_1\sqcup X_2\sqcup \cdots\sqcup X_s
\]
denote the partition of $\mu^{-1}(0)$ into orbit types, let $H_i$ be the isotropy group of some point in $X_i$ for each $i$,
and let $Y_i = (X_i)_{H_i} = X_{H_i}$ denote the points in $X$ with isotropy equal to $H_i$.
Order the $X_i$ so that $H_1 = G$, i.e., $X_1 = Y_1 = \mu^{-1}(0)^G$ is the element of the partition that contains the origin.
As in Section~\ref{sec:SimpCircleAction}, let $D = \{ e^{\sqrt{-1}\,\theta}I\in\GL(V):\theta\in\R\}\simeq\Sp^1$.
Recall that as $D$ commutes with the action of $G$, each $Y_i$ is $D$-invariant, and $D_{|Y_i}$ denotes the subgroup of
$\HOMEO(Y_i)$ given by the restriction of $D$.
Fixing $j > 1$, we claim that $D_{|Y_j}\cap\rho\big(N_G(H_j)\big)_{|Y_j}$ is finite; this is of course clear if $G$ is finite.
Suppose for contradiction that $D_{|Y_j}\cap\rho\big(N_G(H_j)\big)_{|Y_j}$ is infinite. As $D_{|Y_j}\cap\rho\big(N_G(H_j)\big)_{|Y_j}$
is a closed subgroup of $D_{|Y_j}\simeq\Sp^1$, it follows that $D_{|Y_j}\cap\rho_{|Y_j}\big(N_G(H_j)\big) = D_{|Y_j}$ so that
$D_{|Y_j}\leq\rho_{|Y_j}\big(N_G(H_j)\big)$. Let $W$ be the complex linear subspace of $V$ spanned by $Y_j$,
and then $W\subseteq V^{H_j}$. By linearity, it follows that $D_{|W}\subseteq\rho\big(N_G(H_j)\big)_{|W}$.
Let $\mathfrak{n}_j$ denote the Lie algebra of $N_G(H_j)$ and let $\mu_W\colon W\to\mathfrak{n}_j^\ast$ denote the moment map
for the $N_G(H_j)$-action on $W$. Choose a basis for $\mathfrak{n}_j^\ast$ that contains an element dual to the
Lie algebra of the subgroup $D_{|W}\leq N_G(H_j)$ and let $\mu_W^D\colon W\to\R$ denote the component of $\mu_W$ corresponding to this element.
Note that $\mu_W$ is the restriction to $W$ of the moment map of the $N_G(H_j)$-action on $V$ so that for $w\in W$,
$\mu(w) = 0$ implies $\mu_W(w) = 0$. In complex coordinates $(w_1,\ldots,w_k)$ for $W$, the moment map of $D_{|W}$
is given by a scalar multiple of $|w_1|^2 + |w_2|^2 + \cdots + |w_k|^2$, see
\cite[Ex.~2.1]{GuilleminGinzburgKarshon}, \cite[Sec.~2]{HerbigIyengarPflaum}, or \cite[pp.~5]{FarHerSea}.
Therefore, $Y_j$ consists only of the origin, a contradiction. Hence, for each $j > 1$, $D$ acts on $X_j$ with
finite isotropy. By Corollary~\ref{cor:SimpCircleActionDiag}, it follows that
\[
    \chi_\Gamma\big( G\ltimes\mu^{-1}(0)\big)   =   \chi(G\backslash X_1)\chi\big(G\backslash\HOM(\Gamma,G)\big).
\]
As $G$ acts trivially on $V^G$, we have $\mu_{|V^G} = 0$ so that $V^G \subseteq\mu^{-1}(0)$.
Therefore, $X_1 = \mu^{-1}(0)^G = V^G$ is a complex vector space and hence has Euler characteristic $1$, completing the proof.
\end{proof}

\begin{remark}
\label{rem:SympQuotOther}
If $\rho\colon G\to\GL(V)$ is a finite-dimensional linear representation of a compact Lie group $G$, then the proof of Theorem~\ref{thrm:SympQuot}
applies as well with $\mu^{-1}(0)$ replaced by other $G$- and $D$-invariant definable subsets $X\subseteq V$. Letting $H_i$, $i=1,\ldots, s$
be a choice of isotropy group from each conjugacy class, we require that if $N_G(H_i)$ contains a subgroup isomorphic to $\Sp^1$ that
acts on $V^{H_i}$ as scalar multiplication (up to a finite kernel), then $X\cap V_{H_i} = \emptyset$. If this holds,
then Equation~\eqref{eq:SympQuot} does as well with $\mu^{-1}(0)$ replaced with $X$. When $G = \Sp^1$ acts on $V\simeq\C^n$ with weight
vector $(a_1,\ldots,a_n)$, a subset $X\subseteq V$ defined by an equation of the form
\[
    \sum\limits_{i=1}^n b_i |x_i|^{p_i} = 0
\]
with each $p_i\in\Z$ positive satisfies this hypothesis provided that the $b_i$ are nonzero and, if $a_{i_1} = a_{i_2} = \cdots a_{i_k}$ for
some subset $\{i_1,\ldots,i_k\}\subseteq\underline{n}$, then the signs of the $b_{i_j}$, $j=1,\ldots,k$, coincide. The hypothesis
about the signs of the $b_i$ can as well be removed using the argument of Theorem~\ref{thrm:GamECS1Rep}, replacing $D$ with a circle
acting in the coordinates $(x_{i_1},\ldots,x_{i_k})$ with positive and negative weights so that its intersection with the $G$-action is finite.
\end{remark}

Let us illustrate Theorem~\ref{thrm:SympQuot} with the following.

\begin{example}
\label{ex:623SympQuot}
Let $G = \Sp^1$ and $V\simeq\C^3$ with weight vector $(-6,2,3)$.
In coordinates $(x_1, x_2, x_3)$ for $V$, a generating set of the real invariants of the action is given by
$m_1 = |x_1|^2$, $m_2 = |x_2|^2$, and $m_3 = |x_3|^2$, along with the real and imaginary parts of
$p_1 = x_1 x_3^2$, $p_2 = x_1 x_2^3$, and $p_3 = x_2^3 \overline{x_3}^2$. We have
\[
    \mu^{-1}(0) = \{ (x_1, x_2, x_3)\in\C^3 : 6m_1 = 2m_2 + 3m_3 \}
\]
so that $m_1$ is redundant on the shell. The symplectic quotient can be identified with the image of the Hilbert embedding
$H\colon\mu^{-1}(0)\to\R^8$ given by
\[
    H(x_1,x_2,x_3)  =   \big( m_2, m_3, \Real(p_1), \Imaginary(p_1), \Real(p_2), \Imaginary(p_2), \Real(p_3), \Imaginary(p_3)\big).
\]
On the shell, the invariants satisfy relations generated by the following, which we express in complex coordinates for brevity:
\begin{alignat*}{3}
 &2 m_2 m_3^2 + 3 m_3^3 - 6 p_1 \overline{p_1},
    &&6 p_1 \overline{p_2} - 2 m_2 \overline{p_3} - 3 m_3 \overline{p_3},\qquad\qquad
    &&m_2^3 p_1 - p_2 \overline{p_3}
    \\
 &6 \overline{p_1} p_2 - 2 m_2 p_3 - 3 m_3 p_3,
    &&m_3^2 \overline{p_2} - \overline{p_1} \overline{p_3},
    &&m_3^2 p_2 - p_1 p_3,\\
 &2 m_2^4 + 3 m_2^3 m_3 - 6 p_2 \overline{p_2},
    &&m_2^3 \overline{p_1} - \overline{p_2} p_3,\\
 &27 m_3^5 - 24 m_2^2 p_1 \overline{p_1} + 36 m_2 m_3 p_1 \overline{p_1} &&- 54 m_3^2 p_1 \overline{p_1} + 8 p_3 \overline{p_3}.
\end{alignat*}
These relations define the Zariski closure of the image of the Hilbert embedding in $\R^8$. The Jacobian of the relations has
rank $8$ at points where all coordinates $x_i$, and hence all invariants, are nonzero.

Define the sets
\begin{align*}
    &X_1 = \{0\}, \qquad
    X_2 = \{ (x_1,x_2,0)\in\mu^{-1}(0) : x_1, x_2\neq 0 \}, \\
    &X_3 = \{ (x_1,0,x_3)\in\mu^{-1}(0): x_1, x_3\neq 0 \}, \\
    &X_4 = \{ (x_1,x_2,x_3)\in\mu^{-1}(0): x_1, x_2, x_3\neq 0\}.
\end{align*}
Points in $X_2$ have isotropy $R(2)$ and satisfy $m_3 = p_1 = p_3 = 0$. These points are singular, as
the Jacobian of the relations at these points has rank $5$. The image under $H$ of $X_2$ consists of points
$\big( m_2, 0, 0, 0, \Real(p_2), \Imaginary(p_2), 0, 0 \big)\in\R^8$,
such that $m_2\geq 0$ and $3(\Real(p_2)^2 + \Imaginary(p_2)^2) = m_2^4$, hence $\Sp^1\backslash X_2$ is homeomorphic to $\C\smallsetminus\{0\}$.
Points in $X_3$, have isotropy $R(3)$ and satisfy $m_2 = p_2 = p_3 = 0$. These points are as well singular, as
the Jacobian of the relations at these points has rank $6$. The image under $H$ of $X_3$ consists of points
$\big( 0, m_3, \Real(p_1), \Imaginary(p_1), 0, 0, 0, 0 \big)\in\R^8$
with $m_1\geq 0$ and $2(\Real(p_2)^2 + \Imaginary(p_2)^2) = m_1^3$, as well homeomorphic to $\C\smallsetminus\{0\}$.
Points in $X_4$ have trivial isotropy, and their image under $H$ is the collection of nonsingular points in $\Sp^1\backslash\mu^{-1}(0)$,
homeomorphic to $(\C\smallsetminus\{0\})^2$. Hence, in this case, one can verify from this description that
$\chi_\Gamma\big( \Sp^1\ltimes\mu^{-1}(0)\big) = \chi\big(\Sp^1\backslash\HOM(\Gamma,\Sp^1)\big)$ for each finitely presented $\Gamma$.
\end{example}

Many low-dimensional examples of linear symplectic quotients by a positive-dimensional compact Lie group $G$ admit homeomorphisms to
symplectic orbifolds that preserve the decomposition into orbit types
and many other structures; see \cite{GotayBos,LermanMontgomerySjamaar,FarHerSea}. By the results of \cite{HerbigSchwarzSeaton},
this phenomenon does not occur for larger representations; see also \cite{HerbigSchwarzSeaton4}. However, because the $\Gamma$-Euler
characteristics depend on the isotropy groups, they distinguish between the $G$-symplectic quotient and the symplectic orbifold in many cases. We
illustrate this with the following example; see \cite[Prop.~5.3 and 5.5]{HerbigSchwarzSeaton}.

\begin{example}
\label{ex:OrbifoldSympQuot}
Let $G = \SU(2)$, let $R_d$ denote the $(d+1)$-dimensional irreducible representation of $\SU(2)$ on binary forms of degree $d$, and
let $\mu_{R_d}$ denote the moment map of $R_d$.
Then by \cite[Prop.~5.3 and 5.5]{HerbigSchwarzSeaton}, the symplectic quotient corresponding to $R_3$ is homeomorphic to the quotient of
$W_3 = \C$ by $\langle i\rangle\simeq \Z/4\Z$ acting as scalar multiplication, and the symplectic quotient corresponding to $R_4$ is
homeomorphic to the quotient of $W_4 = \C^2$ by the
symmetric group $\mathcal{S}_3$ acting diagonally on two copies of the standard representation on $\C\simeq\R^2$. Each of these homeomorphisms
is shown to in fact be a graded regular symplectomorphism, see the reference for the definition, and hence
preserves the partition into orbit types. Therefore, in the case of $R_3$, the orbit space of $\SU(2)\ltimes \mu_{R_3}^{-1}(0)$ is partitioned into
the origin and a space homeomorphic to $\C\smallsetminus\{0\}$, while in the case of $R_4$, $\vert\SU(2)\ltimes \mu_{R_r}^{-1}(0)\vert$ is partitioned
into the origin, a space homeomorphic to $\C\smallsetminus\{0\}$, and a space finitely covered by $\C^2$ with three complex $1$-dimensional
subspaces removed. Except for the origin, each of these orbit types has vanishing Euler characteristic so
that for any finitely presented $\Gamma$, we confirm the result of Theorem~\ref{thrm:SympQuot} that
\[
    \chi_\Gamma\big(\SU(2)\ltimes \mu_{R_3}^{-1}(0)\big)
        =   \chi_\Gamma\big(\SU(2)\ltimes \mu_{R_4}^{-1}(0)\big)
        =   \chi\big(\SU(2)\backslash\HOM(\Gamma,\SU(2))\big).
\]
In the same way, this description and Theorem~\ref{thrm:SympQuot} yield
\[
    \chi_\Gamma\big(\langle i\rangle\ltimes W_3\big)
        =   \chi\big((\Z/4\Z)\backslash\HOM(\Gamma,\Z/4\Z)\big),
\]
and
\[
    \chi_\Gamma\big(\mathcal{S}_3\ltimes W_4\big)
        =   \chi\big(\mathcal{S}_3\backslash\HOM(\Gamma,\mathcal{S}_3)\big).
\]
It is clear that the $\chi_\Gamma$ of the orbifolds differ from those of the corresponding $\SU(2)$-symplectic quotients;
when $\Gamma = \Z$, we have
$\chi\big(\SU(2)\backslash\HOM(\Z,\SU(2))\big) = 1$,
$\chi\big(\Z/4\Z\backslash\HOM(\Z,\Z/4\Z)\big) = 4$,
and
$\chi\big(\mathcal{S}_3\backslash\HOM(\Z,\mathcal{S}_3)\big) = 3$.
\end{example}

\begin{example}
\label{ex:NonOrbifoldSympQuot}
For an example with $G = \SU(2)$ that is not graded regularly symplectomorphic to an orbifold,
let $V = 2R_2$ using the notation of Example~\ref{ex:OrbifoldSympQuot};
see the proof of \cite[Prop.~5.1]{HerbigSchwarzSeaton} and \cite[Sec.~2, 6.2]{CapeHerbigSeaton}.
The kernel of the action is $\{\pm 1\}$, and $V$ can be identified with $4\R^3$ where $\SU(2)/\{\pm 1\}\simeq\SO(3)$ acts as the defining representation
on each copy of $\R^3$.
Using coordinates $(\boldsymbol{x}_1,\boldsymbol{x}_2,\boldsymbol{x}_3,\boldsymbol{x}_4)$ for $V$ where each $\boldsymbol{x}_i\in\R^3$,
the invariants are given by $u_{ij} = \langle\boldsymbol{x}_i,\boldsymbol{x}_j\rangle$ for $i \leq j$. The shell is given by
\[
    \mu_{2R_2}^{-1}(0) = \big\{ (\boldsymbol{x}_1,\boldsymbol{x}_2,\boldsymbol{x}_3,\boldsymbol{x}_4)\in 4\R^4 :
                            \boldsymbol{x}_1\wedge\boldsymbol{x}_2 + \boldsymbol{x}_3\wedge\boldsymbol{x}_4 = \boldsymbol{0}\}.
\]
The Hilbert embedding $H\colon\mu_{2R_2}^{-1}(0)\to\R^{10}$ is given by
\[
    H(\boldsymbol{x}_1,\boldsymbol{x}_2,\boldsymbol{x}_3,\boldsymbol{x}_4)
        =   (u_{11}, u_{12},u_{13}, u_{14}, u_{22}, u_{23}, u_{24}, u_{33}, u_{34}, u_{44}),
\]
and the Zariski closure of the image of $H$ is defined by the 9 relations
\begin{alignat*}{2}
&u_{14}u_{23}-u_{13}u_{24}+u_{34}^2-u_{33}u_{44},
    &&u_{14}u_{22}-u_{12}u_{24}+u_{24}u_{34}-u_{23}u_{44},
\\
&u_{13}u_{22}-u_{12}u_{23}+u_{24}u_{33}-u_{23}u_{34},
    &&u_{12}u_{14}-u_{11}u_{24}+u_{14}u_{34}-u_{13}u_{44},
\\
&u_{12}u_{13}-u_{11}u_{23}+u_{14}u_{33}-u_{13}u_{34},
    &&u_{12}^2-u_{11}u_{22}-u_{34}^2+u_{33}u_{44},
\\
&u_{24}^2u_{33}-2u_{23}u_{24}u_{34}+u_{22}u_{34}^2+u_{23}^2u_{44}
    &&-u_{22}u_{33}u_{44},
\\
&u_{14}u_{24}u_{33}-2u_{13}u_{24}u_{34}+u_{12}u_{34}^2+u_{34}^3
    &&+u_{13}u_{23}u_{44}-u_{12}u_{33}u_{44}-u_{33}u_{34}u_{44},
\\
&u_{14}^2u_{33}-2u_{13}u_{14}u_{34}+u_{11}u_{34}^2+u_{13}^2u_{44}
    &&-u_{11}u_{33}u_{44}.
\end{alignat*}
The orbit types are given by $X_1 = \{\boldsymbol{0},\boldsymbol{0},\boldsymbol{0},\boldsymbol{0}\}$ with isotropy $\SU(2)$,
the set $X_2$ of $(\boldsymbol{x}_1,\boldsymbol{x}_2,\boldsymbol{x}_3,\boldsymbol{x}_4)$ such that the $\boldsymbol{x}_i$
are parallel and not all $\boldsymbol{0}$ fixed by subgroups isomorphic to $\Sp^1$, and the remainder $X_3$ of the shell with
isotropy $\{\pm 1\}$. The rank of the Jacobian of the relations is $9$ generically, but $3$ on $X_2$ so that $X_2$ consists
of singular points.
\end{example}

\begin{example}
\label{ex:S1SympQuot}
Let $G = \Sp^1$, let $\rho\colon\Sp^1\to\U(V)$ be a unitary $\Sp^1$-representation,
and let $\mu_V$ denote the corresponding moment map. Using the notation of Section~\ref{sec:CircleReps},
let $(x_1,\ldots,x_n)$ be coordinates for $V$ with respect to a basis on which the action is diagonal, and let
$(a_1,\ldots,a_n)\in\Z^n$ be the corresponding weight vector. Then the shell $\mu_V^{-1}(0)$ is the set of points such that
$\sum_{i=1}^n a_i |x_i|^2 = 0$, which satisfies the hypotheses of Corollary~\ref{cor:GamECS1Set}. The set of fixed points
$\mu^{-1}(0)^{\Sp^1} = V^{\Sp^1}$ is a complex vector space with Euler characteristic $1$, and for each $i$ such that
$a_i\neq 0$, we have $\mu^{-1}(0)\cap V_{\{i\}} = \emptyset$. Hence Corollary~\ref{cor:GamECS1Set} yields Theorem~\ref{thrm:SympQuot}
in this case.
\end{example}

\begin{example}
\label{ex:O2SympQuot}
Let $G = \OO(2)$, let $\rho\colon\OO(2)\to\U(V)$ be a unitary $\OO(2)$-representation with $V^{\OO(2)} = \{0\}$ as in \eqref{eq:genO2Rep},
and let $\mu_V$ denote the corresponding moment map.
Using the description in the proof of Theorem~\ref{thrm:GamECO2Rep}, one may similarly describe a direct computation of
$\chi_\Gamma\big(\OO(2)\ltimes\mu_V^{-1}(0)\big)$. In the coordinates used in Theorem~\ref{thrm:GamECO2Rep}, the shell is given by
\[
    \mu_V^{-1}(0) = \left\{ \boldsymbol{v}\in V : \sum\limits_{i=1}^n \alpha_i(|x_{i,1}| - |x_{i,2}|) = 0\right\}.
\]
The sets $\boldsymbol{X}_i$ and $\boldsymbol{X}_I^\times$ do not intersect the shell, and the sets $\boldsymbol{X}_I^\ast$ and $\boldsymbol{Y}_J$
are completely contained in the shell so that the argument of Theorem~\ref{thrm:GamECO2Rep} applies. Similarly, the intersections of the sets
$\boldsymbol{X}_I$ and $\boldsymbol{X}_{I,J}$ with the shell are invariant under the action of the circle by scalar multiplication so that the
same argument demonstrates that these sets have orbit spaces with vanishing Euler characteristic. One may also describe the
$\boldsymbol{X}_I^\ast$, $\boldsymbol{Y}_J$, $\boldsymbol{X}_I$, and $\boldsymbol{X}_{I,J}$ explicitly; there are several cases to consider,
but the orbit space of each admits a description with either $\C\smallsetminus\{0\}$ or $\Sp^1$ as a factor, and hence has vanishing Euler
characteristic.
\end{example}

To conclude, let us briefly recall how real linear symplectic quotients form local models for symplectic quotients of manifolds;
see \cite[Sec.~6]{HerbigSchwarz} and \cite[Sec.~2]{SjamaarLerman}.
Suppose $(M,\omega)$ is a symplectic manifold equipped with the Hamiltonian action of a compact Lie group $G$ with moment map
$\mu_M\colon M\to\mathfrak{g}^\ast$. If $x\in M$ such that $\mu_M(x) = 0$, then by the local normal form theorem for the moment map
\cite{GuilleminSternbergNormalForm,MarleModeleHamiltonianAction}, a $G$-invariant neighborhood $U$ of the $G$-orbit of $x$ in $M$ is $G$-equivariantly
symplectomorphic to $Y = G\times_{G_x} (\mathfrak{m}^\ast\times V)$ where $\mathfrak{m}$ is a complement to the
Lie algebra $\mathfrak{g}_x$ of $G_x$ in $\mathfrak{g}$ as $G_x$-representations and $V$ is the \emph{symplectic slice} at $x$, a subspace
of the slice representation at $x$ for the $G$-action on $M$. Then $V$ can be equipped with a compatible complex structure with respect to which
it is a unitary $G_x$-representation as above. Let $\mu_V\colon V\to\mathfrak{g}_x^\ast$ denote the moment map for $V$ and
$\mu_Y\colon Y\to \mathfrak{g}^\ast$ the moment map for $Y$. Using the results of \cite[pp.384--7]{SjamaarLerman},
the embedding of $\mu_V^{-1}(0)$ into $\mu_Y^{-1}(0)$ induces a weak equivalence
$G_x\ltimes\mu_V^{-1}(0) \to G\ltimes\mu_Y^{-1}(0)$, see \cite[Sec.~2.1]{FarsiSeatonEC}. It follows that the translation groupoid
$G_x\ltimes\mu_V^{-1}(0)$ presenting the real linear symplectic quotient at level $0$ associated to $V$ is Morita equivalent as a topological groupoid
to the translation groupoid $G\ltimes \big(U\cap\mu_M^{-1}(0)\big)$ presenting a neighborhood of the orbit of $x$ in the symplectic quotient
$\mu_M^{-1}(0)/G$. In particular, $\chi_\Gamma\Big(G\ltimes \big(U\cap\mu_M^{-1}(0)\big)\Big) = \chi_\Gamma\big(G_x\ltimes\mu_V^{-1}(0)\big)$,
and the latter can be computed using Theorem~\ref{thrm:SympQuot}. See \cite[Def.~58, Rem.~59]{MetzlerTopSmoothStacks} for the definition of Morita
equivalence and \cite[Cor.~4.15]{FarsiSeatonEC}.


\bibliographystyle{amsplain}
\bibliography{FMS-EulerO2}

\end{document}